\documentclass[reqno]{amsart}
\usepackage{amsfonts,amssymb,amsmath,latexsym,indentfirst}
\usepackage{fontenc,enumerate}
\usepackage{graphics}
\usepackage{color,graphicx}
\usepackage{epstopdf}
\usepackage[pdftex]{hyperref}
\usepackage{latexsym,wasysym,mathrsfs,enumerate}
\DeclareMathAlphabet{\mathpzc}{OT1}{pzc}{m}{it}

\begin{document}

\pagenumbering{arabic}
\title[\hfilneg \hfil The Shrira equation]
{Two dimensional solitary waves in shear flows}

\author[A. Esfahani and  A. Pastor \hfil\hfilneg]
{}

\date{}

\maketitle
\begin{center}
 {\large \textbf{Amin Esfahani}}\\ {\small School of Mathematics and Computer Science\\
Damghan University, Damghan 36715-364, Iran
\\
E-mail: amin@impa.br, esfahani@du.ac.ir }\\
 \end{center}
\begin{center}
 {\large \textbf{Ademir Pastor}}\\ {\small IMECC--UNICAMP\\
Rua S\'ergio Buarque de Holanda, 651\\ Cidade Universit\'aria,
13083-859, Campinas--SP, Brazil.\\
 E-mail: apastor@ime.unicamp.br}\\
 \end{center}

\begin{abstract}
In this paper we study existence and asymptotic behavior of solitary-wave solutions for the generalized Shrira equation, a  two-dimensional model appearing in shear flows. The method used to show the existence of such special solutions is based on the mountain pass theorem. One of the main difficulties consists in showing the compact embedding of the energy space in the Lesbesgue spaces; this is dealt with interpolation theory. Regularity and decay properties of the solitary waves are also established.
\end{abstract}

\numberwithin{equation}{section}

\newtheorem{theorem}{{Theorem}}[section]
\newtheorem{lemma}[theorem]{{\sc \textbf{Lemma}}}
\newtheorem{definition}[theorem]{{ \textbf{Definition}}}
\newtheorem{proposition}[theorem]{Proposition}
\newtheorem{remark}[theorem]{{\sc \textbf{Remark}}}
\newtheorem{corollary}[theorem]{{\sc \textbf{Corollary}}}
\newcommand{\rr}{{\mathbb R}}
\newcommand{\ee}{{\rm e}}
\newcommand{\ii}{{\rm i}}
\newcommand{\rrt}{{{\mathbb R}^2}}
\newcommand{\dd}{{d}}
\newcommand{\lt}{{L^2(\mathbb{R}^2)}}
\newcommand{\T}{{\mathbb T}}
\newcommand{\C}{{\mathbb C}}
\newcommand{\N}{{\mathbb N}}
\newcommand{\lam}{{\lambda}}
\newcommand{\ff}{\varphi}
\newcommand{\ti}{\widetilde}
\newcommand{\what}{\widehat}
\newcommand{\hd}{D^{1/2}_x}
\newcommand{\nhd}{D^{-1/2}_x}
\newcommand{\om}{\omega}
\newcommand{\dr}{S_r(\xi)}
\newcommand{\eps}{\epsilon}
\newcommand{\dk}{S_k}
\newcommand{\A}{\mathscr{A}}
\newcommand{\h}{\mathscr{H}}
\newcommand{\z}{\mathscr{Z}}
\newcommand{\e}{\mathscr{E}}
\newcommand{\f}{\mathscr{F}}
\newcommand{\fim}{\hfill$\square$\\ \\}

\section{Introduction}

Shear flows appear in natural and engineering environments, and in many physical situations. It is connected with a shear stress in solid mechanics, and with the flow induced by a force in fluid dynamics, for instance. In this context, the evolution of essentially two-dimensional weakly long waves is, usually, described by simplified models using the paraxial approximation.
In \cite{shrira}, Shrira described a model for the propagation of nonlinear long-wave perturbations on the background of a boundary-layer type plane-parallel shear flow without inflection points. Within the model, the amplitude $v$ of the longitudinal velocity of the fluid is governed by the equation
\begin{equation}\label{generalshrira}
v_t+\gamma v_x+\alpha vv_x-\beta Q(v_x)=0,
\end{equation}
where $\alpha$, $\beta$, and $\gamma$ are parameters expressed through a profile of the shear flow and $Q$ is the Cauchy-Hadamard integral transform given by
$$
Q(f)(x,y)=\frac{1}{2\pi}{\rm p.v.}\int_{\rr^2}\frac{f(z,w)}{\big((x-z)^2+(y-w)^2\big)^{3/2}}dzdw.
$$
The model also describes the amplitude of the perturbation of the horizontal
velocity component of a sheared flow of electrons (see \cite{GZ}). For additional applications see also \cite{ass}, \cite{pelinovsky-shrira},  \cite{pe-st}.

When considering nearly one-dimensional waves, in the dimensionless form\-  $u(x',y',t')=\alpha v(x-ct,\sqrt{2}y,\beta t)/2\beta$, equation \eqref{generalshrira} can be reduced to the \textit{so called} Shrira equation (see \cite{pelinovsky-shrira})
\begin{equation}\label{shrira-original}
u_t-\mathscr{H}\Delta u+2uu_x=0,
\end{equation}
where we omitted the primes. Here, $\Delta$ denotes the two-dimensional Laplacian operator and $\mathscr{H}$ is the Hilbert transform defined by
\[
\mathscr{H}u(x,y,t)=\frac{1}{\pi}{\rm p.v.}\int_\rr\frac{u(z,y,t)}{x-z} dz.
\]
In particular, at least from the mathematical viewpoint, equation \eqref{shrira-original} can be seen as a two-dimensional extension of the well-known Benjamin-Ono (BO) equation
\begin{equation}\label{boequation}
u_t-\mathscr{H}\partial_x^2 u+2uu_x=0, \quad x,t\in\rr.
\end{equation}
The study of multi-dimensional extension of BO equation has received considerable attention in recent years (see e.g., \cite{cp1}, \cite{cp2}, \cite{ep-ill}, \cite{EP1}, \cite{EP3}, \cite{epb}, \cite{Latorre}, \cite{maris}, \cite{rb}, and references therein). However,  when a suitable result is available for \eqref{boequation}, is not completely clear how to extend such a result for \eqref{shrira-original} and, in general, it demands extra efforts.

To the best of our knowledge, not so much is known about equations \eqref{generalshrira} and \eqref{shrira-original} and a few papers are available in the current literature.  In particular, numerical results concerned with the instability of one-dimensional solitons of the BO equation were presented in \cite{pelinovsky-shrira} and \cite{GZ}. It is to be observed that these equations presents a strong anisotropic character of the dispersive part, which turns out to be one of the main difficulties to be dealt with.

In this paper we are interested in studying existence, regularity and decay properties of solitary waves for the generalized  two-dimensional Shrira equation, namely,
\begin{equation}\label{main-shrira}
u_t-\mathscr{H}\Delta u+(f(u))_x=0,
\end{equation}
where $u=u(x,y,t)$, $(x,y)\in\rrt$, $t\geq0$, and $f$ is a real-valued continuous function. Observe that if $f(u)=u^2$ then equation \eqref{main-shrira}  reduces to \eqref{shrira-original}.

For one hand, from the physical point of view, a solitary wave is a wave that propagates without changing its profile along the temporal evolution, usually it has one global peak and decays far away from the peak. On the other hand, from the mathematical point of view, a solitary wave is a special solution of a partial differential equation belonging to a very particular function space. Nowadays, existence and properties of solitary waves  have become one of the major issues in PDEs. The interested reader will find a large number of good texts in the current literature, which we refrain from list them at this moment.

The solitary-wave solutions of interest in the context of equation \eqref{main-shrira} have the form $$u(x,y,t)=\ff(x-ct,y),$$ where $c> 0$ is the wave speed and $\ff$ is a real-valued function with a suitable decay at infinity. Substituting this form into \eqref{main-shrira}, it transpires that $\ff$ must satisfy the nonlinear  equation
\begin{equation}\label{shrira-0}
-c\ff_x-\mathscr{H}\Delta\ff+(f(\ff))_x=0,
\end{equation}
where we have replaced the variable $z=x-ct$ by $x$. This last equation can be rewritten in the following form
\begin{equation}\label{shrira}
-c\ff-\mathscr{H}\partial_x^{-1}\Delta\ff+f(\ff)=0.
\end{equation}
Hence, in order to show the existence of solitary waves, it suffices to show that \eqref{shrira-0}, or equivalently \eqref{shrira}, has a solution.

\begin{remark}\label{scalingremark}
Note that the wave speed $c$   can be normalized to $1$  at least if $f$ is homogeneous of degree $p+1$ such as $f(u)=u^{p+1}$. Indeed, the scale change
\begin{equation}  \label{normalized}
\phi(x,y)=a\ff\left(bx,dy\right),
\end{equation}
transforms \eqref{shrira-0} in $\ff$, into the same in $\phi$, but with $c=1$, where $a=c^{-1/p}$ and $b=d=c^{-1}$.
\end{remark}

By multiplying \eqref{shrira} by $\varphi$ and integrating by parts, one sees that the
 natural space to study \eqref{shrira} is
\[
\z=\left\{u\in L^2(\rrt);\;\hd u,\;\nhd u_y\in L^2(\rrt)\right\},
\]
equipped with the norm
\[
\|u\|_\z=\left(\int_\rrt(cu^2+|\hd u|^2+|\nhd u_y|^2)\dd x\;\dd y\right)^{1/2},
\]
where $c>0$ is a fixed constant and $D_x^{\pm1/2}$    denotes   the fractional derivatives operator of order $\pm1/2$ with respect to $x$, defined via Fourier transform by $({D^{\pm1/2}_xu})^{\wedge}(\xi,\eta)=|\xi|^{\pm1/2}\widehat{u}(\xi,\eta)$. Note that $\z$ is a Hilbert space with the  scalar product
\[
(u,v)_\z=\int_\rrt(cuv+\hd u\hd v+\nhd u_y\nhd v_y)dxdy.
\]
Also, $\z$ is an anisotropic space including fractional negative derivatives, which, for one hand, brings many difficulties when applying analytical methods.

The paper is organized as follows. We start Section \ref{sec2} by showing a suitable Gagliardo-Nirenberg-type inequality, which ensures that the space $\z$ is embedded in suitable Lebesgue spaces. In order to show that indeed we have a compact embedding we use interpolation theory. Thus, we are able to see $\z$ as an interpolation space between two compatible pair of Hilbert spaces containing only integer derivatives. With the compact embedding in hand, we  use the mountain pass theorem without the Palais-Smale condition, in order to show the existence of at least one nontrivial solution. Variational characterizations of such a solutions are also provided. In Section \ref{sec3} we study regularity and decay properties of the solitary waves. Our regularity results is based on the so called Lizorkin lemma, which gives a sufficient condition to a function be a Fourier multiplier on $L^p(\rr^n)$. The decay properties are obtained once we write the equation as a convolution equation and get some suitable estimates on the corresponding kernel. Of course, the anisotropic structure of the kernel also brings many technical difficulties. Finally, in Section \ref{sec4} we present a nonexistence result of positive solitary waves.

The issue of the orbital stability/instability of solitary waves of \eqref{main-shrira} will be studied in a future work.\\

\noindent {\bf Notation.} Otherwise stated, we follow the standard notation in PDEs. In particular, we use $C$ to denote several positive constants that may vary from line to line. In order to simplify notation in some places where the constant is not important, if $a$ and $b$ are two positive numbers, we use $a\lesssim b$ to mean that there exists a positive constant $C$ such that $a\leq Cb$. By $L^p=L^p(\rrt)$ we denote the standard Lebesgue space. Sometimes we use subscript to indicate which variable we are concerned with; for instance, $L^p_x=L^p_x(\rr)$ means the space $L^p(\rr)$ with respect to the variable $x$; thus given a function $f=f(x,y)$, the notation $\|f\|_{L^p_x}$ means we are taking the $L^p$ norm of $f$ only with respect to $x$. Also, if no confusion is caused, we use $\int_{\rrt}f$ instead of $\int_{\rrt}f(x,y)\dd x\dd y$.

\section{Existence of Solitary Waves}\label{sec2}

In this section we provide the existence of solitary-wave solutions for \eqref{main-shrira}. As we already said, our main tool in doing so will be mountain pass theorem.

\subsection{A Gagliardo-Nirenberg-type inequality and embeddings}

First, we are going to obtain an embedding of the space $\z$, which is  appropriate to study equation \eqref{shrira}. For the sake of simplicity, in this subsection, we assume that the constant $c$ in the definition of $\z$ is normalized to 1.

\begin{lemma}[Gagliardo-Nirenberg-type inequality]\label{embed}
 Assume $0\leq p\leq2$. Then there is a constant $C>0$ (depending only on $p$) such that for any $\ff\in\mathscr{Z}$,
\begin{equation}\label{GN}
\|\ff\|_{L^{p+2}}^{p+2}\leq C\|\ff\|_{L^2}^{2-p}\left\|D_x^{-1/2}\ff_y\right\|_{L^2}^{p/2}\left\|D_x^{1/2}\ff\right\|_{L^2}^{3p/2}.
\end{equation}
As a consequence, it follows that there is a constant $C>0$ such that for all $\ff\in\mathscr{Z}$,
\[
\|\ff\|_{L^{p+2}}\leq C\|\ff\|_{\mathscr{Z}},
\]
which is to say $\mathscr{Z}$ is  continuously embedded in $L^{p+2}$.
\end{lemma}
\begin{proof}
It suffices to assume $0<p\leq2$.  The lemma will be established only for $C_0^\infty(\rrt)$ functions; a standard limit method then can be used to complete the proof. By the Gagliardo-Nirenberg inequality (see, for instance, \cite{ABLS} or \cite{hmoli}), there exists $C>0$ such that, for all $g\in H^{1/2}(\rr)$,
\[
\|g\|_{L^{p+2}(\rr)}\leq C\|g\|_{L^2(\rr)}^{\frac{2}{p+2}}\|D_x^{1/2}g\|_{L^2(\rr)}^{\frac{p}{p+2}}.
\]
Assume for the moment that $0<p<2$. Integrating on the $y$ variable, it follows that
\begin{equation}\label{ineq-0}
\begin{split}
\|\ff\|_{L^{p+2}(\rrt)}^{p+2}&\leq C\int_{\rr}\|\ff(\cdot,y)\|_{L^2(\rr)}^2\;\|D_x^{1/2}\ff(\cdot,y)\|_{L^2(\rr)}^p\;\dd y\\
&\leq C \left\|\|\ff(\cdot,y)\|^2_{L^2(\rr)}\right\|_{L^{\frac{2}{2-p}}(\rr)}\;
\|D_x^{1/2}\ff\|_{L^2(\rrt)}^p\\
&\leq C\|\ff\|_{L^2(\rrt)}^{2-p}\;
\sup_{y\in\rr}\|\ff(\cdot,y)\|_{L^2(\rr)}^p\;
\|D_x^{1/2}\ff\|_{L^2(\rrt)}^p.
\end{split}
\end{equation}
We now estimate the middle term in \eqref{ineq-0}.  Fixed $y\in\rr$, from H\"older's inequality, we deduce
\begin{equation}
\begin{split}
\|\ff(\cdot,y)\|_{L^2(\rr)}^2&
=2\int_{\rr}\int_{-\infty}^y\ff(x,z)\ff_y(x,z)\dd z\dd x\\
&=2\int_{-\infty}^y\int_{\rr}\hd\ff(x,z)\nhd\ff_y(x,z)\dd x\dd z\\
&\leq 2\int_\rrt |\hd\ff(x,y)||\nhd\ff_y(x,y)| \dd x\dd y\\
&\leq 2\|\hd\ff\|_{L^2(\rrt)}\|\nhd\ff_y\|_{L^2(\rrt)}.
\end{split}
\end{equation}
As a consequence,
\begin{equation}\label{enq-1}
\sup_{y\in\rr}\|\ff(\cdot,y)\|_{L^2(\rr)}\leq C\|D_x^{1/2}\ff\|_{L^2(\rrt)}^{1/2}\|\nhd\ff_y\|_{L^2(\rrt)}^{1/2}.
\end{equation}
A combination of \eqref{enq-1} with \eqref{ineq-0} yields the first statement if $0<p<2$. For $p=2$, from the first inequality in \eqref{ineq-0} and \eqref{enq-1}, we deduce
\begin{equation*}
\begin{split}
\|\ff\|_{L^{4}(\rrt)}^{4}&\leq C\int_{\rr}\|\ff(\cdot,y)\|_{L^2(\rr)}^2\;\|D_x^{1/2}\ff(\cdot,y)\|_{L^2(\rr)}^2\;\dd y\\
&\leq C
\sup_{y\in\rr}\|\ff(\cdot,y)\|_{L^2(\rr)}^2\;
\|D_x^{1/2}\ff\|_{L^2(\rrt)}^2\\
&\leq C\|\nhd\ff_y\|_{L^2(\rr)}\|\hd\ff_y\|_{L^2(\rr)}^3,
\end{split}
\end{equation*}
which is the desired. The lemma is thus proved.
\end{proof}
\begin{remark}
An argument  similar to that in Lemma \ref{embed} gives the continuous embedding $\z\hookrightarrow L_y^qL_x^r(\rrt)$, for any $q,r\geq2$ satisfying $\frac{1}{q}+\frac{1}{r}\geq\frac{1}{2}$.
\end{remark}

As is well known in the theory of critical point, in order to rule out the trivial solution, a compactness result is usually necessary. Here, we will prove the following.

\begin{lemma}[Compact embedding]\label{complemma}
If $0\leq p<2$ then the embedding $\z\hookrightarrow L^{p+2}_{loc}(\rrt)$ is compact.
\end{lemma}

Due to the anisotropic property of $\z$ involving negative derivatives, some difficulties appear in the proof of Lemma \ref{complemma}. To do so, we will identify $\z$ as an interpolation space by using the real interpolation method.

For any real number $s\geq0$, we introduce the space
\[
X^s:=\Big\{f\in \mathcal{S}'(\rrt); \; (1+|\xi|+|\xi|^{-1}\eta^2)^{s}\widehat{f}\in L^2(\rrt)\Big\}.
\]
The space $X^s$ is a Hilbert spaces endowed with the scalar product
$$
(f,g)_{X^s}:=\int_\rrt(1+|\xi|+|\xi|^{-1}\eta^2)^{2s}\widehat{f}(\xi,\eta)\,\overline{\widehat{g}}(\xi,\eta)\;d\xi d\eta.
$$
In particular, from Plancherel's theorem we have $X^0=L^2$ and $X^{1/2}=\z$. If  $0\leq s_1\leq s_2$ then $X^{s_1}\subset X^{s_2}$. The space $X^1$ is a suitable space that involves only integer derivatives and, moreover, it can be defined as the closure of $\partial_x(C_0^\infty(\rr^2))$ for the norm (see \cite{dbs1})
$$
\|\ff_x\|_{X^1}=\left( \|\ff_x\|_{L^2}^2+\|\ff_{xx}\|_{L^2}^2+\|\ff_{yy}\|_{L^2}^2\right)^{1/2}.
$$
So, our idea is to look $\z$ as an interpolation space between $L^2$ and $X^1$.

In what follows, if $(H_0,H_1)$ is a compatible pair of Hilbert spaces and $\theta\in(0,1)$, we denote by $(H_0,H_1)_{\theta}$ the  space $(H_0,H_1)_{\theta,2}$. Here, for $q\in(1,\infty)$, $(H_0,H_1)_{\theta,q}$  denotes the intermediate space with respect to the couple $(H_0,H_1)$ using either the $J$-method or the $K$-method (see e.g., \cite{bergh}, \cite{mclean}, or \cite{triebel}). For our purposes, the following results will be useful.

\begin{lemma}\label{lemmaintb2}
Let $(X_0,X_1)$ and $(Y_0,Y_1)$ be two compatible pair of Hilbert spaces. Then, for $0<\theta<1$, $((X_0,X_1)_\theta, (Y_0,Y_1)_\theta)$ is a pair of interpolation spaces with respect to $((X_0,X_1), (Y_0,Y_1))$, which is exact of exponent $\theta$.

In particular, if $A$ is a bounded linear operator from $X_0$ to $Y_0$ and from $X_1$ to $Y_1$, then it is also bounded from $(X_0,X_1)_\theta$ to $(Y_0,Y_1)_\theta$.
\end{lemma}
\begin{proof}
See, for instance, \cite[Chapter 1]{triebel} or \cite[Appendix B]{mclean}.
\end{proof}

\begin{lemma}\label{chandlemma}
Let $(H_0,H_1)$ be a compatible pair of Hilbert spaces,
let $(X,\mathcal{M},\mu)$ be a measure space and let $\mathcal{Y}$ denote the set of measurable
functions from $X$ to $\C$. Suppose that there exist a linear map $\mathcal{A}:H_0+H_1\to \mathcal{Y}$ and,
for $j= 0, 1$, functions $w_j \in \mathcal{Y}$, with  $w_j > 0$ almost everywhere, such that the
mappings $\mathcal{A}:H_j\to L^2(X, \mathcal{M},w_j\mu)$ are unitary isomorphisms. For $\theta\in(0,1)$, define
$$
H_\theta=\Big\{ \phi\in H_0+H_1; \|\phi\|_{H_\theta}:=\left( \int_X |w_\theta \mathcal{A}\phi|^2\;d\mu  \right)^{1/2}<\infty \Big\},
$$
where $w_\theta=w_0^{1-\theta}w_1^\theta$. Then $H_\theta=(H_0,H_1)_\theta$ with
equality of norms.
\end{lemma}
\begin{proof}
See \cite[Corollary 3.2]{chan}.
\end{proof}

As an application of the above lemma, we have.

\begin{lemma}\label{interplemma}
The space $\z$ is such that
$$
\z=X^{1/2}=(X^1,L^2)_{1/2}.
$$
\end{lemma}
\begin{proof}
It suffices to apply Lemma \ref{chandlemma} with $X=\rr^2$, $w_0=1$, $w_1=(1+|\xi|+|\xi|^{-1}\eta^2)$, and $\mathcal{A}$ being the Fourier transform.
\end{proof}

For any open set $\Omega\subset\rr^2$ and $s\geq0$, we define
$$
X^s(\Omega):=\Big\{u\in L^2(\Omega); \;u=f|_\Omega, \;\mbox{for\;\;some}\;f\in X^s\Big\}.
$$
Endowed with the norm
$$
\|u\|_{X^s(\Omega)}:=\inf\Big\{\|f\|_{X^s};\;u=f|_\Omega\;\mbox{with}\; f\in X^s \Big\},
$$
the space $X^s(\Omega)$ is a Hilbert space.

The next step is the construction of an extension operator from $X^1(\Omega)$ to $X^1$, where $\Omega$ is a rectangle. This construction was essentially given in \cite{lopez}, but for the sake of completeness we bring the details.

\begin{lemma}\label{extesionop}
Let $\Omega=(a,b)\times(c,d)$. Then, there exists a bounded (extension) linear operator, say, $E$, from $X^1(\Omega)$ to $X^1$ such that, for any $u\in X^1(\Omega)$, $Eu=u$ in $\Omega$, $\|Eu\|_{L^2}\leq C\|u\|_{L^2(\Omega)}$ and  $\|Eu\|_{X^1}\leq C\|u\|_{X^1(\Omega)}$, where $C$ is a constant depending only on $\Omega$.
\end{lemma}
\begin{proof}
Take any $u\in X^1(\Omega)$ and, without loss of generality, assume that $u=\partial_xf$  in $\Omega$, for some smooth function $f\in C_0^\infty(\rr^2)$ with $\|\partial_xf\|_{X^1}\leq 2\|u\|_{X^1(\Omega)}$. Define
$$
f_0(x,y)=f(x,y)-\frac{1}{b-a}\int_a^bf(x,y)\,dx.
$$
In $\Omega$ it is clear that $u=\partial_xf_0$. From Poincar\'e's inequality,
$$
\int_a^b\left|f(x,y)-\int_a^bf(z,y)\,dz \right|^2\,dx\leq (b-a)^2\int_a^b|\partial_xf(x,y)|^2\,dx.
$$
Hence, integrating with respect to $y$ on $(c,d)$,
\begin{equation}\label{fel5}
\|f_0\|_{L^2(\Omega)}\leq (b-a)\|\partial_xf\|_{L^2(\Omega)}=(b-a)\|u\|_{L^2(\Omega)}.
\end{equation}

Now we extend $f_0$ to the rectangle $[2a-b,2b-a]\times[c,d]$ by using a generalized reflection argument. Indeed, let
\begin{displaymath}
f_1(x,y) = \left\{ \begin{array}{ll}
f_0(x,y), & {\rm if}\; x\in[a,b],\\
{\displaystyle \sum_{i=1}^4a_if_0\Big(\frac{i+1}{i}b-\frac{1}{i}x,y\Big)}, & \textrm{if}\; x\in[b,2b-a],\\
{\displaystyle\sum_{i=1}^4a_if_0\Big(\frac{i+1}{i}a-\frac{1}{i}x,y\Big),} & \textrm{if}\; x\in[2a-b,a],
\end{array} \right.
\end{displaymath}
where the coefficients $a_i$ are such that
\begin{displaymath}
\left\{ \begin{array}{ll}
a_1+a_2+a_3+a_4=1,\\\\
{\displaystyle a_1+\frac{a_2}{2}+\frac{a_3}{3}+\frac{a_4}{4}=-1},\\\\
{\displaystyle a_1+\frac{a_2}{4}+\frac{a_3}{9}+\frac{a_4}{16}=1},\\\\
{\displaystyle a_1+\frac{a_2}{8}+\frac{a_3}{27}+\frac{a_4}{64}=-1}.\\\\
\end{array} \right.
\end{displaymath}
It is clear that $f_1$ is a $C^2$ function on $(2a-b,2b-a)\times(c,d)$ with
$$
\|\partial^\alpha f_1\|_{L^2((2a-b,2b-a)\times(c,d))}\leq C\|\partial^\alpha f_0\|_{L^2(\Omega)},
$$
for all multi-indices $\alpha\in \N^2$ with $|\alpha|\leq2$.
By using the same argument we can extend $f_1$ to the rectangle $\widetilde{\Omega}=(2a-b,2b-a)\times(2c-d,2d-c)$ by defining a $C^2$ function $f_2$ such that
\begin{equation}\label{fel6}
\|\partial^\alpha f_2\|_{L^2(\widetilde{\Omega})}\leq C\|\partial^\alpha f_0\|_{L^2(\Omega)},
\end{equation}
for all multi-indices $\alpha\in \N^2$ with $|\alpha|\leq2$.

Now take a smooth function $\eta$ such that $\eta\equiv1$ on $\Omega$ and $\eta\equiv0$ on $\rr^2\setminus\widetilde{\Omega}$. Finally, define the extension operator $E$ by setting $Eu=\partial_x(\eta f_2)$. Let us estimate $Eu$ in the  $X^1$ norm. First of all, note that from \eqref{fel5} and \eqref{fel6}, we have
\begin{equation}\label{fel7}
\begin{split}
\|Eu\|_{L^2}&\leq C\Big( \|f_2\|_{L^2(\widetilde{\Omega})}+\|\partial_x f_2\|_{L^2(\widetilde{\Omega})} \Big)\\
&\leq C\Big( \|f_0\|_{L^2({\Omega})}+\|\partial_x f_0\|_{L^2({\Omega})} \Big)\\
&\leq C\|u\|_{L^2({\Omega})}.
\end{split}
\end{equation}
Also, by using \eqref{fel7} and \eqref{fel6},
\begin{equation}\label{fel8}
\begin{split}
\|\partial_xEu\|_{L^2}&\leq C\Big( \|f_2\|_{L^2(\widetilde{\Omega})}+\|\partial_x f_2\|_{L^2(\widetilde{\Omega})}+\|\partial_x^2f_2\|_{L^2(\widetilde{\Omega})} \Big)\\
&\leq C\Big(\|u\|_{L^2({\Omega})}+ \|\partial_x(\partial_xf_0)\|_{L^2({\Omega})} \Big)\\
&\leq C\Big(\|u\|_{L^2({\Omega})}+\|\partial_xu\|_{L^2({\Omega})}\Big).
\end{split}
\end{equation}
It remains to estimate $\partial_x^{-1}\partial_y^2Eu$. In this case, we have
\begin{equation}\label{fel9}
\begin{split}
\|\partial_x^{-1}\partial_y^2Eu\|_{L^2}&\leq C\Big( \|f_2\|_{L^2(\widetilde{\Omega})}+\|\partial_y f_2\|_{L^2(\widetilde{\Omega})}+\|\partial_y^2f_2\|_{L^2(\widetilde{\Omega})} \Big)\\
&\leq C\Big(\|u\|_{L^2({\Omega})}+\|\partial_y f
_0\|_{L^2({\Omega})}+ \|\partial_y^2f_0\|_{L^2({\Omega})} \Big)\\
&\leq  C\Big(\|u\|_{L^2({\Omega})}+\|\partial_y f
_0\|_{L^2({\Omega})}+ \|\partial_x^{-1}\partial_y^2u\|_{L^2({\Omega})} \Big).
\end{split}
\end{equation}
Note that
$$
\partial_yf_0(x,y)=\int_c^d\partial_y^2f_0(x,z)\,dz.
$$
Hence, from the Cauchy-Schwarz inequality,
$$
|\partial_yf_0(x,y)|^2\leq (y-c)\int_c^d|\partial_y^2f_0(x,z)|^2\,dz.
$$
This last inequality now implies
\begin{equation}\label{fel10}
\|\partial_yf_0\|_{L^2(\Omega)}\leq \frac{d-c}{\sqrt2}\|\partial_y^2f_0\|_{L^2(\Omega)}\leq \frac{d-c}{\sqrt2}\|\partial_x^{-1}\partial_y^2u\|_{L^2({\Omega})}
\end{equation}
Gathering together \eqref{fel9} and \eqref{fel10}, we obtain
$$
\|\partial_x^{-1}\partial_y^2Eu\|_{L^2}\leq  C\Big(\|u\|_{L^2({\Omega})}+ \|\partial_x^{-1}\partial_y^2u\|_{L^2({\Omega})} \Big).
$$
The proof of the lemma is thus completed.
\end{proof}

\begin{remark}\label{remarkC}
A simple inspection in the proof of Lemma \ref{extesionop} reveals that the positive constant $C$ depend only on the difference $b-a$ and $d-c$, but not on the rectangle $\Omega$ itself.
\end{remark}

With the extension operator in hand, we can also prove that the space $X^{1/2}(\Omega)$ is also the interpolation of $L^2(\Omega)$ and $X^1(\Omega)$. Results of this type are well known in the context of the standard Sobolev spaces, see, for instance, \cite[Lemma 4.2]{chan}.

\begin{lemma}\label{interplem1}
	Let $\Omega=(a,b)\times(c,d)$. Then, $X^{1/2}(\Omega)=(L^2(\Omega),X^1(\Omega))_{1/2}$, with equivalence of norms.
\end{lemma}
\begin{proof}
From Lemma \ref{interplemma} we know that $X^{1/2}=(X^1,L^2)_{1/2}$. Note that, for any $s\geq0$, the restriction operator  $R:X^s\to X^s(\Omega)$ is bounded. Thus, from Lemma  \ref{lemmaintb2},
\begin{equation}\label{fel1}
X^{1/2}(\Omega)=R(X^{1/2})\subset (L^2(\Omega),X^1(\Omega))_{1/2}.
\end{equation}
On the other hand, the extension operator $E$ constructed in Lemma  \ref{extesionop} is bounded from $L^2(\Omega)$ to $L^2$ and from $X^1(\Omega)$ to $X^1$. Thus, another application of Lemma \ref{lemmaintb2} gives $E((L^2(\Omega),X^1(\Omega))_{1/2})\subset (X^1,L^2)_{1/2}=X^{1/2}$. Hence,
\begin{equation}\label{fel2}
(L^2(\Omega),X^1(\Omega))_{1/2}=RE((L^2(\Omega),X^1(\Omega))_{1/2})\subset R(X^{1/2})=X^{1/2}(\Omega).
\end{equation}
A combination of \eqref{fel1} and \eqref{fel2} yields the desired.
\end{proof}

\begin{proposition}\label{prop2.9}
Let $\{\Omega_i\}_{i\in \N}$ be a covering of $\rr^2$, where $\Omega_i$ is an open square with edges parallel to the coordinate axis and side-length $\ell$, and such that each point of $\rr^2$ is contained in at most three squares. Then, there exists a constant $C>0$, such that
$$
\sum_{i=0}^\infty\|u\|_{X^{1/2}(\Omega_i)}^2\leq C\|u\|_{X^{1/2}}^2,
$$
for any $u\in X^{1/2}$.
\end{proposition}
\begin{proof}
Let $E_i$ be the extension operator from $X^1(\Omega_i)$ to $X^1$ as constructed in Lemma \ref{extesionop}. Thus, from Lemma \ref{extesionop},
\begin{equation}\label{fel11}
\|E_iu\|^2_{X^1}\leq C \int_{\Omega_i}(|u|^2+|\partial_x u|^2+|\partial_x^{-1}\partial_y^2 u|^2).
\end{equation}
As observed in Remark \ref{remarkC}, the constant $C$ in \eqref{fel11} depends only on $\ell$ but not on $i\in\N$. By observing that the restriction operator $R_i:X^1\to X^1(\Omega_i)$ is bounded with norm 1 and the composition $R_iE_i$ is the identity operator, we obtain
\begin{equation}\label{fel12}
\|u\|_{X^1(\Omega_i)}\leq \|E_iu\|_{X^1}.
\end{equation}
Hence, \eqref{fel11} and \eqref{fel12} imply
$$
\sum_{i=0}^\infty\|u\|^2_{X^1(\Omega_i)}\leq C\sum_{i=0}^\infty\int_{\Omega_i}(|u|^2+|\partial_x u|^2+|\partial_x^{-1}\partial_y^2 u|^2)\leq 3C\|u\|^2_{X^1}.
$$
This means that the restriction operator is bounded from $X^1$ to $\ell^2(X^1(\Omega_i))$. On the other hand, the trivial inequality,
$$
\sum_{i=0}^\infty\|u\|_{L^2(\Omega_i)}\leq 3\|u\|_{L^2},
$$
implies that the restriction operator is also bounded from $L^2$ to $\ell^2(L^2(\Omega_i))$. Then, Theorem 1.18.1 in \cite{triebel} combined with Lemmas \ref{lemmaintb2}, \ref{interplemma}, and \ref{interplem1} gives that the restriction is bounded from $X^{1/2}$ to $\ell^2(X^{1/2}(\Omega_i))$, which is the desired conclusion.
\end{proof}

Now we are ready to prove Lemma \ref{complemma}.

\begin{proof}[Proof of Lemma \ref{complemma}]
Let  $\{\ff_n\}$ be a bounded sequence in $\mathscr{Z}=X^{1/2}$ and select a constant $C_0>0$ such that $\|\ff_n\|_\mathscr{Z}\leq C_0$. It is sufficient to show that  $\{\ff_n\}$ has a convergent subsequence in $L^2_{loc}(\rr^2)$, because if this is true then Lemma \ref{embed} implies that  $\{\ff_n\}$ also has a convergent subsequence in $L^{p+2}_{loc}(\rr^2)$, $0<p<2$. To do that, it suffices to show that $\{\ff_n\}$ converges, up to a subsequence, in $L^2(\Omega_R)$, where $\Omega_R$ is a square with center at the origin, edges parallel to the coordinate axis,  and side-length $R>0$. Let $E_R$ be the extension operator constructed in Lemma \ref{extesionop}. By construction, if $u\in X^{1/2}$ then $E_R(u)=u$ in $\Omega_R$ and $E_R(u)=0$ in $\rr^2\setminus\Omega_{3R}$. Thus, without loss of generality, we can assume that $\ff_n=E_R(\ff_n)$ for all $n\in\N$. Now, since $X^{1/2}$ is a Hilbert space, there exists $\ff\in X^{1/2}$ such that $\ff_n\rightharpoonup\ff$ weakly in $X^{1/2}$. In addition, replacing $\ff_n$ by $\ff_n-\ff$, if necessary, we can assume  $\ff=0$, that is, $\ff_n\rightharpoonup0$ in $X^{1/2}$.

Fixed $\rho>0$ to be chosen later, define
\[\begin{split}
&Q_0=\{(\xi,\eta)\in\rr^2;\,|\xi|\leq\rho,\;|\eta|\leq\rho\},\\
&Q_1=\{(\xi,\eta)\in\rr^2;\,|\xi|\leq\rho,\;|\eta|\geq\rho\},\\
&Q_2=\{(\xi,\eta)\in\rr^2;\,|\xi|\geq\rho\}.
\end{split}\]
Plancherel's identity and the fact that $\ff_n=0$ outside the square $\Omega_{3R}$  yield
\begin{equation}\label{prob-in}
\int_{\Omega_{3R}}|\ff_n|^2= \int_{\rr^2}|\ff_n|^2 =\int_{\rr^2}|\what{\ff}_n|^2=\sum_{i=0}^2\int_{Q_i}|\what{\ff}_n|^2.
\end{equation}

From the definitions of $Q_1$ and $Q_2$, it is clear that
\[
\int_{Q_1}|\what{\ff}_n|^2
=\int_{Q_1}\frac{|\xi|}{|\eta|^2}\left|\widehat{D_x^{-1/2}\partial_y\ff}_n\right|^2
\leq\frac{1}{\rho}\left\|D_x^{-1/2}\partial_y\ff_n\right\|_{L^2(\rr^2)}
\leq \frac{C_0}{\rho}
\]
and
\[
\int_{Q_2}|\what{\ff}_n|^2
=\int_{Q_2}\frac{1}{|\xi|}\left|\what{D_x^{1/2}\ff}_n\right|^2
\leq\frac{1}{\rho}\left\|D_x^{1/2}\ff_n\right\|_{L^2(\rr^2)}
\leq \frac{C_0}{\rho}.
\]
Fix $\varepsilon>0$; then choosing $\rho>0$ sufficiently large leads to
\[
\int_{Q_1}|\what{\ff}_n|^2+\int_{Q_2}|\what{\ff}_n|^2\leq\varepsilon/2.
\]
Since $\ff_n\rightharpoonup0$ in $L^2(\rr^2)$, then $\what{\ff}_n$ tends to zero as $n\to\infty$ and
\begin{equation}\label{prob1-in}
|\what{\ff}_n(\xi,\eta)|\leq\|\ff_n\|_{L^1(\Omega_{3R})}\leq C\|\ff_n\|_{L^2}.
\end{equation}
 Lebesgue's dominated convergence theorem implies that
\[
\lim_{n\to\infty}\int_{Q_0}|\what{\ff}_n|^2=0.
\]
 Thus we have proved that, up to a subsequence, $\ff_n\to0$ in $L^2_{{\rm loc}}(\rr^2)$, which concludes the proof of the lemma.
\end{proof}

We conclude this section by observing that Lemma \ref{embed} also holds when norms are restrict to a rectangle.

\begin{lemma}\label{boundemb}
Assume $0\leq p\leq 2$. Let $\Omega=(a,b)\times(c,d)$ be a rectangle. There exist a constant $C>0$ such that, for any $\ff\in X^{1/2}(\Omega)$,
$$
\|\ff\|_{L^{p+2}(\Omega)}\leq C\|\ff\|_{X^{1/2}(\Omega)}.
$$
\end{lemma}
\begin{proof}
From Lemmas \ref{extesionop} and \ref{lemmaintb2} we know that the extension operator is bounded from $X^{1/2}(\Omega)$ to $X^{1/2}$. Now it suffices to note that the identity operator is continuous from $X^{1/2}$ to $L^{p+2}(\rrt)$ and the restriction operator is continuous from $L^{p+2}(\rrt)$ to $L^{p+2}(\Omega)$.
\end{proof}

\subsection{Pohojaev-type identities and nonexistence of solitary waves}

As usual, let us first to get an insight for which class of nonlinearities, solutions of \eqref{shrira} are expected. This is done with integration by parts.

\begin{theorem}\label{nonexistence}
Assume $c>0$. Equation  \eqref{main-shrira} does not possesses solitary-wave solutions of the form $u(x,y,t)=\ff(x-ct,t)$, $\ff\in\z$, whether
\begin{enumerate}[(i)]
\item ${\displaystyle \int_{\rr^2} \ff f(\ff)\;\dd x\dd y\leq2\int_{\rr^2}F(\ff)\;\dd x\dd y}$; or
\item  ${\displaystyle \int_{\rr^2} \left(\ff f(\ff)+2F(\ff)\right)\;\dd x\dd y\leq0}$.
\end{enumerate}
\end{theorem}
\begin{proof} Formally, by multiplying   equation   \eqref{shrira} by $\ff$ and $y\ff_y$, respectively, and integrating over $\rr^2$, we deduce the identities
\begin{gather}
\int_{\rr^2}\left[-c\ff^2-\ff\mathscr{H}\ff_x- (D_x^{-1/2}\ff_y)^2
+\ff f(\ff)\right]\;\dd x\dd y=0\label{shr-nonex-1},\\
\int_{\rr^2}\left[c\ff^2+\ff\mathscr{H}\ff_x- (D_x^{-1/2}\ff_y)^2-
2F(\ff)\right]\;\dd x\dd y=0.\label{shr-nonex-3}
\end{gather}
For smooth functions decaying to $0$ at infinity, these
formulas follow from integration by parts together with elementary
properties of the Hilbert transform. The identities can be justified
for functions of the minimal regularity required for them to make
sense by the truncation argument put forward in \cite{dbs1}. The proof is completed by subtracting  and adding \eqref{shr-nonex-1} and \eqref{shr-nonex-3}.
 \end{proof}

\begin{remark}
Unfortunately,  Theorem \ref{nonexistence} is not strong enough to rule out the existence of solitary waves even in the case of a power-law nonlinearity. This is mainly because, in view of the nonlocal operator $\h$,  we are not able to prove a Pohojaev-type identity on the $x$-variable for \eqref{shrira}, i.e., (see \cite{epb} for similar calculations)
\begin{equation}\label{shr-nonex-2}
\int_{\rr^2}\left[c\ff^2+2 \psi_y ^2-2F(\ff)\right]\;\dd x\dd y=0.
\end{equation}
Indeed, if \eqref{shr-nonex-2} were valid. Then
subtracting  \eqref{shr-nonex-2} and \eqref{shr-nonex-3} leads to
\begin{equation}
\int_{\rr^2}\left(\ff\mathscr{H}\ff_x
-3\psi_y^2\right)\;\dd x\dd y=0\label{shr-nonex-4}.
\end{equation}
Adding   \eqref{shr-nonex-1} and
\eqref{shr-nonex-3}, there appears
\begin{equation}
\int_{\rr^2}\left(-2\psi_y^2+\ff f(\ff)-2F(\ff)\right)
\;\dd x\dd y=0\label{shr-nonex-5}.
\end{equation}
Finally, plugging \eqref{shr-nonex-5} in
\eqref{shr-nonex-2}, there obtains
\begin{equation}
c\int_{\rr^2}\ff^2\;\dd x\dd y=\int_{\rr^2}\left(4F(\ff)-\ff f(\ff)\right)\;\dd x\dd y.
\end{equation}
Therefore, there would exist no nontrivial   solitary-wave solution of \eqref{main-shrira} provided
\begin{equation}\label{shr-nonex-6}
4\int_{\rr^2} F(\ff)\;\dd x\dd y\leq\int_{\rr^2}\ff f(\ff)\;\dd x\dd y.
\end{equation}
To fix ideas, if we assume $f(\ff)=\ff^{p+1}$ and that  $\int\ff^{p+2}\geq0$, \eqref{shr-nonex-6} implies that solitary waves do not exist if $p>4$. This seems to be consistent with our embedding in Lemma \ref{embed}.
\end{remark}

\subsection{Existence of solitary waves}

In this subsection we will prove the existence of solution for \eqref{shrira} under suitable conditions on the nonlinearity $f$. Having in mind Lemma \ref{embed}, we assume the following.\\

\begin{enumerate}[(${\rm A}_1)$]
	\item $f:\rr\to\rr$ is continuous and $f(0)=0$;
	\item There exists $C>0$ such that $|f(u)|\leq C(|u|+|u|^{p-1})$, $p\in (2,4)$ and $f(u)=o(|u|)$ as $|u|\to0$;
	\item There exists $\mu>2$ such that $0<\mu F(u)\leq uf(u)$ for every $u\in\rr$, where $F$ is the primitive function of $f$.
	\item There exists $\omega\in\z$     such that $\lambda^{-2}F(\lambda\omega)\to+\infty$ as $\lambda\to+\infty$.\\
\end{enumerate}

The above assumptions are the ones suitable to apply minimax theory (see e.g. \cite{wi}). Probably, assumptions $(A_1)-(A_4)$ can be weakened to establish the existence of solitary waves. However, since our main interest is the study of \eqref{main-shrira} with a power-law nonlinearity, this will be enough to our purposes.

We start with the following vanishing property.

\begin{lemma}\label{cpt}
If $\{u_n\}$ is a bounded sequence in $\z$ and there is $r>0$ such that
\[
\lim_{n\to+\infty}\sup_{(x,y)\in\rrt}\int_{B_r(x,y)}|u_n|^2\dd x\dd y=0,
\]
then, for $2<p<4$,
$$\lim_{n\to\infty}\|u_n\|_{L^p(\rrt)}=0,$$
 where $B_r(x,y)\subset\rrt$ is the open ball centered at $(x,y)$ with radius $r$.
\end{lemma}
\begin{proof}
Let $\{\Omega_i\}_{i\in \N}$ be a covering of $\rr^2$, where $\Omega_i$ is an open square with edges parallel to the coordinate axis and side-length $r$, and such that each point of $\rr^2$ is contained in at most three squares. By the H\"{o}lder inequality and Lemma \ref{boundemb}, there
holds, for any $u\in\z=X^{1/2}$,
\[\begin{split}
\|u\|_{L^3(\Omega_i)}^3&\leq \|u\|_{L^2(\Omega_i)}
\|u\|_{L^4(\Omega_i)}^2\\
&\leq\|u\|_{L^2(\Omega_i)}
\|u\|_{X^{1/2}(\Omega_i)}^2.
\end{split}\]
Thus, in view of Proposition \ref{prop2.9},
\[
\|u_n\|_{L^3(\rrt)}^3\lesssim \sum_{i=0}^\infty\int_{\Omega_i}|u_n|^3dxdy\lesssim
\sup_{(x,y)\in\rrt}
\|u_n\|_{L^2(B_r(x,y))}\|u_n\|_\z^2.
\]
Since $\{u_n\}$ is bounded, the assumption implies that $u_n\to0$ in $L^3(\rrt)$. Finally, by  interpolation
and Lemma \ref{embed}, there are $\theta_1,\theta_2\in(0,1)$, such that
\begin{equation}\label{cpt1}
\|u_n\|_{L^p(\rrt)}\leq\|u_n\|_{L^2(\rrt)}^{\theta_1}\|u_n\|_{L^3(\rrt)}^{1-\theta_1}\lesssim \|u_n\|_\z^{\theta_1}\|u_n\|_{L^3(\rrt)}^{1-\theta_1}, \qquad p\in(2,3)
\end{equation}
and
\begin{equation}\label{cpt2}
\|u_n\|_{L^p(\rrt)}\leq\|u_n\|_{L^4(\rrt)}^{\theta_2}\|u_n\|_{L^3(\rrt)}^{1-\theta_2}\lesssim \|u_n\|_\z^{\theta_2}\|u_n\|_{L^3(\rrt)}^{1-\theta_2}, \qquad p\in(3,4)
\end{equation}
The fact that $u_n\to0$ in $L^3(\rrt)$ and \eqref{cpt1}-\eqref{cpt2} then implies that $u_n\to0$ in $L^p(\rrt)$, for all $p\in(2,4)$, and the proof of the lemma is complete.
\end{proof}

Now we are able to prove our main theorem in this section.

\begin{theorem}[Existence]\label{existence}
	Assume $c>0$.
Under assumptions (A$_1$)-(A$_4$), equation \eqref{shrira-0} possesses a nontrivial
solution $\ff\in\z$.
\end{theorem}
\begin{proof}
	We will use the well known mountain pass lemma without the Palais-Smale condition (see \cite{ar}).
Let
\begin{equation}
S(u)=\frac{1}{2}\|u\|_\z^2-\int_\rrt F(u)\;\dd x\dd y
\end{equation}
and note that critical points of $S$ are weak solutions of \eqref{shrira}.

We claim that, for some constant $C_0>0$,
\begin{equation}\label{limFt}
|F(u)|\leq \frac{1}{4}|u|^2+C_0|u|^p.
\end{equation}
Indeed, from assumption  (A$_2$), there exists $\varepsilon>0$ such $|f(u)|\leq\frac{1}{2} |u|$, when $|u|\leq\varepsilon$. Hence, in this case
\begin{equation}\label{limF2}
|F(u)|=\left|\int_0^uf(s)ds\right|\leq \frac{1}{4}|u|^2.
\end{equation}
On the other hand, choose a constant $\widetilde{C}>0$ such that $(1/\widetilde{C})^{1/(p-2)}\leq\varepsilon$. Thus, if $|u|\geq\varepsilon$, we immediately see that $|u|^2\leq \widetilde{C}|u|^p$. Hence, in this case,
\begin{equation}\label{limFp}
|F(u)|=\left|\int_0^uf(s)ds\right|\leq C(|u|^2+|u|^p)\leq C_0|u|^p.
\end{equation}
Collecting \eqref{limF2} and \eqref{limFp} yield \eqref{limFt}.

Now, an application of Lemma \ref{embed} gives, for any $u\in\z$,
\[
S(u)\geq\frac{1}{2}\|u\|_\z^2-\int_{\rr^2}(\frac{1}{4}|u|^2+C_0|u|^p)\dd x\dd y
\geq\frac{1}{4}\|u\|^2_\z-C_1\|u\|_\z^p,
\]
where $C_1>0$. Hence, there are $\delta>0$, independent of $u$, and $r>0$ small enough with the property that $S(u)\geq\delta$ if $\|u\|_\z=r$. On the other hand, it follows from assumption  (A$_4$) that $S(\lambda u)\to-\infty$ as $\lambda\to+\infty$. Thus there exists $e_1\in\z$ such that $\|e_1\|_\z>r$ and $S(e_1)<0$.

Let $d$ be the mountain-pass level, that is,
\[
d=\inf_{\gamma\in\Gamma}\max_{t\in[0,1]}S(\gamma(t))
\]
where
\[
\Gamma=\{\gamma\in C([0,1];\z);\;\gamma(0)=0,\;S(\gamma(1))<0\}.
\]
Clearly $d\geq \inf_{\|u\|_\z=r}S(u)>0$. Therefore, from the Mountain-Pass Lemma without the Palais-Smale condition there is a sequence $\{u_n\}\subset\z$ such that $S'(u_n)\to0$ and $S(u_n)\to d$, as $n\to+\infty$ (see e.g. \cite[Theorem 2.9]{wi}). For $n$ large enough, we obtain from assumption (A$_3$) that
\[
\left(\frac{1}{2}-\frac{1}{\mu}\right)\|u_n\|_\z^2\leq S(u_n)-\frac{1}{\mu}\langle S'(u_n),u_n\rangle
\leq d+o(1)+\|u_n\|_\z.
\]
Since $\mu>2$, we obtain that $\{u_n\}$ is bounded.

We now claim that there is no $r>0$ such that
\begin{equation}\label{vanishing1}
\lim_{n\to+\infty}\sup_{(x,y)\in\rrt}\int_{B_{r}(x,y)}|u_n|^2\dd x\dd y\to0.
\end{equation}
Indeed, assume the contrary, that is, \eqref{vanishing1} holds for some $r'>0$.
Then, from Lemma \ref{cpt},
\begin{equation}\label{contrdp}
\|u_n\|_{L^p}\to0, \qquad \mbox{for}\;\;\; p\in(2,4),
\end{equation}
and there is a sequence $\epsilon_n\to0$ such that
\begin{equation}\label{vanis1}
\begin{split}
d&=S(u_n)-\frac{1}{2}\langle S'(u_n),u_n\rangle_{L^2}+\epsilon_n\\
&=\int_{\rr^2}\left(\frac{1}{2}f(u_n)u_n-F(u_n)\right)\dd x\dd y+\epsilon_n\\
&\lesssim \|u_n\|_{L^2}^2+\|u_n\|_{L^p}^p.
\end{split}
\end{equation}
Since $d>0$, taking the limit in \eqref{vanis1}, we get a contradiction with \eqref{contrdp}.

 Therefore, by selecting if necessary a subsequence, we can assume that   there is a sequence $(x_n,y_n)\subset\rrt$ such that
\[
\|u_n\|_{L^2(B_1(x_n,y_n))}^2\geq \varrho/2>0,\qquad {\mbox{for all}}\;n,
\]
where
\[
\varrho=\lim_{n\to\infty}\sup_{(x,y)\in\rrt}\int_{B_1(x,y)}|u_n|^2\;\dd x\dd y\neq0.
\]
Then the functions $\ff_n(x,y)=u_n(x+x_n,y+y_n)$ satisfy
\begin{equation}\label{vanis3}
\|\ff_n\|_{L^2(B_1(0))}^2\geq \varrho/2>0
\end{equation}
and $\{\ff_n\}$ is bounded in $\z$. Thus, it converges to some $\ff\in\z$ weakly in $\z$ and strongly in $L^2_{{\rm loc}}(\rrt)$, by Lemma \ref{complemma}. From \eqref{vanis3} it is clear that $\ff\neq0$ and for every $\chi\in \z$, we have
\[
\langle S'(\ff),\chi\rangle=\lim_{n\to+\infty}\langle S'(\ff_n),\chi\rangle=0.
\]
This shows that $\ff$ is a nontrivial solution of \eqref{shrira-0} and completes the proof of the theorem.
\end{proof}
\begin{remark}
	To the best of our knowledge, the (non)existence of stationary solutions of \eqref{shrira-0} when $c=0$ and $p=4$ remains as an open problem.
\end{remark}

\subsection{Variational characterization of ground states.}
In this subsection we will show that the solution obtained in Theorem \ref{existence} minimizes some variational problems under the additional assumption:\\

\noindent(A$_5$) The function $t\mapsto t^{-1}\int_{\rrt} uf(tu)\;\dd x\dd y$ is strictly increasing on $(0,+\infty)$ and
$$
\lim_{t\to\infty}t^{-1}\int_{\rrt} uf(tu)\;\dd x\dd y=+\infty.\\
$$

Indeed, let
$$
I(u)=\langle S'(u), u\rangle=\|u\|_\z^2-\int_\rrt uf(u)\;\dd x\dd y.
$$
Consider the Nehari manifold
\[
\tilde{\Gamma}=\{u\in\z;\;I(u)=0,\;u\neq0\},
\]
and the minimization problem
\begin{equation}\label{gamma-tilde}
\tilde{d}=\inf_{u\in\tilde{\Gamma}}S(u)
\end{equation}
Also, let $d^*$ be the minimax value
\begin{equation}\label{gamma-star}
d^\ast=\inf_{u\in \z}\sup_{t\geq0}S(tu),
\end{equation}

\begin{lemma}\label{lema15}
For every $u\in\z\setminus\{0\}$ there exists a unique number $t_u>0$, such that $t_uu\in\tilde\Gamma$ and
$$
S(t_uu)=\max_{t\geq0}S(tu).
$$
In addition, the function $u\mapsto t_u$ is continuous and the map $u\mapsto t_uu$ is an homeomorphism from the unit sphere of $\z$ to $\tilde{\Gamma}$.
\end{lemma}
\begin{proof}
First we note that since
$$
S(tu)=\frac{t^2}{2}\|u\|_{\z}^2-\int_{\rr^2} F(tu)\dd x\dd y
$$
we have
$$
\frac{d}{dt}S(tu)=t\left( \|u\|_{\z}^2-t^{-1}\int_{\rrt} uf(tu)\;\dd x\dd y\right).
$$
Hence, from (A$_5$) the function $t\mapsto \frac{d}{d t}S(tu)=:g(t)$ vanishes at only one point $t_u>0$. In addition, since the function $t\mapsto -t^{-1}\int_{\rrt} uf(tu)\;\dd x\dd y$ is strictly decreasing on $(0,\infty)$, we see that $g(t)>0$ on $(0,t_u)$ and $g(t)<0$ on $(t_u,\infty)$, which means  that $t_u$ is a maximum point for $S(tu)$. The rest of the proof runs, for instance, as in \cite[Lemma 4.1]{wi}); so we omit the details.
\end{proof}

\begin{lemma}\label{equi}
Under the above notation, there hold $d=\tilde d=d^\ast$.
\end{lemma}
\begin{proof}
We divide the proof into some steps.

\noindent{\bf Step 1.} $d\geq\tilde{d}$.

First we see that, as in the proof of Theorem \ref{existence}, $I(u)>0$ in a neighborhood of the origin, except at the origin.  Also,  we have from (A$_3$) that, for $v\in\z$,
\[
\begin{split}
2S(v)&=\|v\|_\z^2-2\int_{\rrt} F(v)\;\dd x\dd y>
\|v\|_\z^2-\mu\int_{\rrt} F(v)\;\dd x\dd y\\
&\geq
\|v\|_\z^2-\int_{\rrt} vf(v)\;\dd x\dd y=I(v).
\end{split}
\]
Now let $\gamma$ be in $\Gamma$.
So $I(\gamma(t))>0$, for small $t$ and $I(\gamma(1))<2S(\gamma(1))<0$. By continuity, $\gamma$ crosses $\tilde\Gamma$, that is, there exists $t_0\in(0,1)$ such that $\gamma(t_0)\in\tilde{\Gamma}$. Consequently, $\tilde d\leq S(\gamma(t_0))\leq\max_{t\in[0,1]}S(\gamma(t))$ and this proves Step 1.\\

\noindent{\bf Step 2.} $d\leq{d}^*$.

For any $u\in\z$, from the proof of Lemma \ref{lema15}, there exists $t_0$ sufficiently large such that $S(t_0u)<0$. By defining $\gamma_0(t)=tt_0u$ we immediately see that $\gamma_0\in\Gamma$. Thus,
$$
d\leq\max_{t\in[0,1]}S(\gamma_0(t))=\max_{t\in[0,1]}S(tt_0u)\leq \max_{t\geq0}S(tu)
$$
The arbitrariness of $u$ gives Step 2.\\

\noindent{\bf Step 3.} $\tilde d={d}^*$.

Given any $u\in\z\setminus\{0\}$ we can fin $t_u>0$ such that $t_uu\in\tilde{\Gamma}$ and
$$
\tilde{d}\leq \inf_{u\in\tilde{\Gamma}}S(u)\leq S(t_uu)=\max_{t\geq0}S(tu).
$$
This shows that $\tilde{d}\leq d^*$. On the other hand, for any $u\in \tilde{\Gamma}$, from Lemma \ref{lema15}, there is $v$ in the unit sphere of $\z$ such that $u=t_vv$. Thus,
$$
\inf_{u\in\z}S(t_uu)\leq S(t_vv)=S(u),
$$
and, consequently, $\inf_{u\in\z}S(t_uu)\leq \tilde{d}.$ At last, the relation
$$
d^*=\inf_{u\in\z}\max_{t\geq0}S(tu)=\inf_{u\in\z}S(t_uu)\leq\tilde{d}
$$
establishes Step 3.

By combining Steps 1,2, and 3, we have $d\leq d^*=\tilde{d}\leq d$ and the proof is completed.
\end{proof}

\begin{definition}
	A solution $\ff\in\z$ of \eqref{shrira-0} is called
	a ground state, if $\ff$ minimizes the action $S$ among all solutions of \eqref{shrira-0}.
\end{definition}

\begin{theorem}\label{existence-1}
Let (A$_1$)-(A$_5$) hold. There exists a minimizer $u\in \tilde\Gamma$ of problem \eqref{gamma-tilde}. In addition, $u$ is a ground state solution.
\end{theorem}
\begin{proof}
As in the proof of Theorem \ref{existence}, we can take a bounded Palais-Smale sequence $\{u_n\}\subset\z$ and a solution $u\in\z\setminus\{0\}$ such that $S(u_n)\to d$, $S'(u_n)\to0$, $S'(u)=0$ and $u_n\to u$ $a.e.$ and in $L^p_{{\rm loc}}(\rrt)$, as $n\to+\infty$. This immediately implies that $u\in\tilde{\Gamma}$ and
\begin{equation}\label{dtildee1}
\tilde{d}=\inf_{v\in\tilde{\Gamma}}S(v)\leq S(u).
\end{equation}
On the other hand, because $I(u_n)\to0$, Lemma \ref{equi} and Fatou's lemma, yield
\begin{equation}\label{dtildee2}
\begin{split}
\tilde{d}&=d=\liminf_{n\to\infty}\left(S(u_n)-\frac{1}{2}I(u_n)\right)=\liminf_{n\to\infty}\int_\rrt\left(\frac{1}{2}u_nf(u_n)-F(u_n)\right)\dd x\dd y\\
& \geq \int_\rrt\left(\frac{1}{2}uf(u)-F(u)\right)\dd x\dd y=S(u)-\frac{1}{2}I(u)=S(u).
\end{split}
\end{equation}
From \eqref{dtildee1} and \eqref{dtildee2} we deduce that $\tilde{d}=S(u)$.  Finally, if $v$ is any critical point of $S$, then $v\in\tilde{\Gamma}$ and  $S(u)\leq S(v)$, which means that $u$ is a ground state.
\end{proof}

\begin{theorem}\label{equi-theo-1}
Let (A$_1$)-(A$_5$) hold. Suppose also that $f\in C^1(\rr)$ and
\begin{equation}\label{new-a5}
\int_\rrt uf(u)\;\dd x\dd y< \int_\rrt u^2f'(u)\;\dd x\dd y.
\end{equation}
 Then for any nonzero $u\in\z$, the following assertions are equivalent:
\begin{enumerate}[(i)]
\item $u$ is a ground state;
\item $I(u)=0$ and $\inf\{G(v);\;v\in\tilde\Gamma\}=\tilde d=G(u)$, where
\[
G(u)=\int_\rrt \left(\frac{1}{2}uf(u)-F(u)\right)\dd x\dd y.
\]
\end{enumerate}
\end{theorem}
\begin{proof}
(i)$\Rightarrow$(ii). If $u$ is a ground state, we have $S'(u)=0$, which implies that $I(u)=0$. On the other hand, for any $u\in\tilde\Gamma$,
\begin{equation}\label{identity-0}
S(u)=S(u)-\frac{1}{2}I(u)=\int_\rrt \left(\frac{1}{2}uf(u)-F(u)\right)\dd x\dd y=G(u).
\end{equation}
Hence,
$$
\tilde{d}=S(u)=\inf_{v\in\tilde{\Gamma}}S(v)=\inf_{v\in\tilde{\Gamma}}G(v).
$$

(ii)$\Rightarrow$(i). Let $u\in\z$ satisfy (ii). Then, by using \eqref{identity-0}, there is a Lagrange multiplier $\theta$ such that $\theta I'(u)=S'(u)$. Therefore,
\[
\theta\langle I'(u),u\rangle=\langle S'(u),u\rangle=I(u)=0.
\]
But,
\[\begin{split}
\langle I'(u),u\rangle&=2\|u\|_\z^2-\int_\rrt f'(u) u^2\;\dd x\dd y-\int_\rrt f(u) u\;\dd x\dd y\\
&=2I(u)+\int_\rrt f(u) u\;\dd x\dd y-\int_\rrt f'(u) u^2\;\dd x\dd y\\
&=\int_\rrt f(u) u\;\dd x\dd y-\int_\rrt f'(u) u^2\;\dd x\dd y\\
&<0,
\end{split}\]
where we used  \eqref{new-a5} in the last inequality. Therefore $\theta=0$ and $S'(u)=0$, which implies that $u$ is a ground state.
\end{proof}

\section{Regularity and Decay}\label{sec3}

In this section we will discuss some regularity and spatially decay properties of solitary waves. For  the simplicity, throughout  the section, we assume  $c=1$ and $f$, satisfies the grow condition $|f(u)|\leq C|u|^{p-1}$, $p\in(2,4)$.

\subsection{Regularity}
The difficulty in studying  regularity properties of the solutions of \eqref{shrira-0} or \eqref{shrira}, comes from the fact that the operator $\mathscr{H}\Delta$ is nonlocal and non-isotropic. Here, we will adopt the strategy put forward in \cite{dbs1} (see also \cite{maris} and \cite{za} for  applications to multi-dimensional models). The following  H\"{o}rmander-Mikhlin type theorem will be useful.

\begin{lemma}[Lizorkin lemma]\label{lizolemma}
Let $\Lambda:\rr^n\to\rr$ be a $C^n$ function for $|\xi_j|>0$,$j=1,\ldots,n$.  Assume that there exists a constant $M>0$ such that
$$
\left| \xi_1^{k_1}\ldots\xi_n^{k_n} \frac{\partial^k\Lambda(\xi)}{\partial\xi_1^{k_1}\ldots \partial\xi_n^{k_n}} \right|\leq M,
$$
where $k_i$ take the values $0$ or $1$ and $k=k_1+\ldots k_n=0,1,\ldots,n$. Then $\Lambda$ is a Fourier Multiplier on $L^q(\rr^n)$, $1<q<\infty$.
\end{lemma}
\begin{proof}
See \cite{lizorkin}.
\end{proof}

Now we can proof the following.

\begin{theorem}[Regularity]\label{regularity}
Assume $p\in(2,4)$.
Any solitary-wave solution $\ff\in\z$ of \eqref{main-shrira} belongs to $W^{1,r}(\rrt)$, where  $r\in(1,\infty)$.
Moreover $\ff\in W^{m+1,r}(\rrt)$, for $m=1,2$, if $f\in C^m(\rr)$. In particular, if $f(u)=u^2$ then $\ff\in H^{\infty}(\rr^2)$.
\end{theorem}
\begin{proof}
 We are left  to prove the regularity result for the nonlinear equation
\begin{equation}\label{elliptic}
\ff_x+\mathscr{H}\Delta\ff=\left(f(\ff)\right)_x.
\end{equation}
Let $\ff\in\z$ be a solution of \eqref{elliptic}.
By Lemma \ref{embed}, one has $\mathscr{Z}\hookrightarrow L^r(\rr^2)$, $r\in[2,4]$, and therefore $f(\ff)\in L^{\frac{r}{p-1}}(\rr^2)$. It can be easily checked  that multipliers $\frac{|\xi|}{|\xi|+\xi^2+\eta^2}$, $\frac{\xi|\xi|}{|\xi|+\xi^2+\eta^2}$ and $\frac{|\xi|\eta}{|\xi|+\xi^2+\eta^2}$ satisfy   the assumptions in Lemma \ref{lizolemma}. Hence   $\ff$, $\ff_x$, $\ff_y\in L^{q}(\rr^2)$, where
\begin{equation}\label{cases-1}
q\in
\begin{cases}
\left[\frac{2}{p-1},\frac{4}{p-1}\right],&p\in(2,3),\\
(1,2],&p=3,\\
\Big(1,\frac{4}{p-1}\Big],&p\in(3,4).
\end{cases}
\end{equation}

We now divide the proof into three cases.

\noindent {\bf Case 1. $2<p<3$.}  From \eqref{cases-1} we see, in particular, that $\ff,\nabla\ff\in L^2(\rr^2)$. In view of the Sobolev embedding  $H^1(\rr^2)\hookrightarrow L^r(\rr^2)$, $r\in[2,\infty)$, we deduce that $\ff\in L^r(\rr^2)$, $r\in[\frac{2}{p-1},\infty)$. As a consequence, $f(\ff)\in L^r(\rr^2)$, $r\in[\frac{2}{p-1},\infty)$. Thus, we can apply Lemma \ref{lizolemma} to conclude that $\nabla\ff\in L^r(\rr^2)$, $r\in[\frac{2}{p-1},\infty)$ and, consequently, $\ff\in W^{1,r}$ with $r\in[\frac{2}{p-1},\infty)$.

Now, let $p_0=\frac{2}{p-1}$ and define $p_1=\frac{p_0}{p-1}$. It is clear that $f(\ff)\in L^{p_1}(\rrt)$ and $p_1\leq1$ if and only if $p\geq\sqrt2+1$.  Hence, if $p\geq\sqrt2+1$ we can apply Lemma \ref{lizolemma} to conclude that $\ff,\nabla\ff\in  L^r(\rr^2)$, $r\in(1,\infty)$. Assume now $p<\sqrt2+1$ and define, inductively, $p_n=\frac{p_{n-1}}{p-1}$. Note that $0<p_n<p_{n-1}$ and $p_n\leq1$ if and only if $p\geq {2}^{\frac{1}{n+1}}+1$. The result then follows, using Lemma \ref{lizolemma} because $p_n\to0$ and $ {2}^{\frac{1}{n+1}}+1\to2$, as $n\to\infty$.

\noindent {\bf Case 2. $p=3$.} Here we also have $\ff,\nabla\ff\in L^2(\rr^2)$. So, as in Case 1 we obtain $\ff,\nabla\ff\in L^r(\rr^2)$, $r\in[2,\infty)$, which combined with \eqref{cases-1} gives the desired.

\noindent {\bf Case 3. $3<p<4$.}
Here we have $\ff, \ff_x, \ff_y\in L^{{\frac{4}{p-1}}}(\rr^2)$. By using the Gagliardo-Nirenberg inequality
\begin{equation}\label{glusual}
\|u\|_{L^s}\leq C\|\nabla u\|_{L^r}^\theta\|u\|_{L^q}^{1-\theta}, \qquad \frac{1}{s}=\frac{1-\theta}{q}+\theta\Big(\frac{1}{r}-\frac{1}{2}\Big), \;\;\theta\in[0,1],
\end{equation}
with $\theta=1$ and $r=\frac{4}{p-1}$, we deduce that $\ff\in  L^{{\frac{4}{p-3}}}(\rr^2)$ and, consequently, $f(\ff)\in L^{\frac{4}{(p-1)(p-3)}}(\rr^2)$. An application of Lemma \ref{lizolemma} yields  $\ff,\nabla\ff\in L^{{\frac{4}{(p-1)(p-3)}}}(\rr^2)$. Now we need to use an iteration process. Indeed, let $p_1$ be the positive root of the equation $(p-1)(p-3)-2=0$, that is, $p_1=2+\sqrt3\in(3,4)$. Since the function $\mu_1(p)=4/(p-1)(p-3)$ is strictly decreasing on the interval $(3,4)$ and $\mu_1(p_1)=2$, we have that $\mu_1(p)\geq2$ in $(3,p_1]$. Consequently, interpolating between $L^{\frac{4}{p-1}}$ and $L^{\mu_1(p)}$, $p\in(3,p_1]$ we obtain $\ff,\nabla\ff\in L^2$. By proceeding as in Case 1 we conclude the result if $p\in(3,p_1]$.

Assume now $p\in(p_1,4)$. Since  $\ff,\nabla\ff\in L^{{\frac{4}{(p-1)(p-3)}}}(\rr^2)$, we can use \eqref{glusual} to conclude that $\ff\in L^{\frac{4}{(p-1)(p-3)-2}}$. It is to be noted that because $p\in(p_1,4)$ we have $(p-1)(p-3)-2>0$. Thus we obtain, $f(\ff)\in L^{\frac{4}{(p-1)^2(p-3)-2(p-1)}}$. Lemma \ref{lizolemma} again implies $\ff,\nabla\ff\in L^{\frac{4}{(p-1)^2(p-3)-2(p-1)}}$. Let $P_2$ be the polynomial $P_2(p)=(p-1)^2(p-3)-2(p-1)-2$. Since $P_2(p_1)=-2$ and $P_2(4)=1$, it follows that $P_2$ has a root in the interval $(p_1,4)$, which we shall call $p_2$. Again, since the function $\mu_2(p)=4/[(p-1)^2(p-3)-2(p-1)]$ is strictly decreasing on the interval $(p_1,4)$ and $\mu_2(p_2)=2$, we deduce that $\mu_2(p)\geq2$ on the interval  $(p_1,p_2)$. Another interpolation gives $\ff,\nabla\ff\in L^2$ and the proof is also completed for $p\in (p_1,p_2]$.

Following this process, we define inductively the polynomial $P_n$ in the following way: having defined $P_{n-1}$, we define $P_n$ by the relation $P_{n}(p)=(p-1)P_{n-1}(p)-2$. Precisely, $P_n$ has the expression
$$
P_n(p)=(p-1)^n(p-3)-2(p-1)^{n-1}-2(p-1)^{n-2}-\ldots-2(p-1)-2.
$$
Also, inductively, if $p_{n-1}$ is the root of $P_{n-1}$ in the interval $(3,4)$, noting that $P_n(p_{n-1})=-2$ and $P_n(4)=1$, we define $p_n$ to be the root of $P_n$ on the interval $(p_{n-1},4)$.
Note $\{p_n\}$ is increasing and bounded by $4$. Thus, in order to complete the proof it suffices to show that the sequence $\{p_n\}$ converges to $4$, as $n\to\infty$. But this follows at once because $P_n(p)\to-\infty$ for any $p\in(0,4)$. This completes the proof in Case 3.

Now suppose that $f$ is $C^1$. Then,  for all $2<p<4$, $f'(\ff)\ff_x\in L^q(\rrt)$, where $1< q<\infty$. On the other hand, \eqref{elliptic} is equivalent to
\[
-\Delta\ff=\h\left(f(\ff)\right)_x-\h\ff_x.
\]
Thus $\Delta\ff\in L^q(\rrt)$ by Riesz's theorem \cite{riesz} (see also \cite{duoandikoetxea}),   where $1< q<\infty$. The proof is now completed by iteration. The case of $f\in C^2$ is similar.
\end{proof}

\begin{remark}\label{vanizero}
It can be seen from Theorem \ref{regularity}, Sobolev's embedding $W^{1,r}(\rrt)\hookrightarrow L^\infty(\rrt)$, $r>2$, and Morrey's inequality that any solitary wave $\ff\in\z$ of \eqref{shrira-0} indeed belongs to $L^\infty(\rrt)\cap C(\rrt)$ and vanishes at infinity.
\end{remark}

Next we prove that the high regularity of $f$ is reflected in the analyticity of the traveling waves.

\begin{theorem}[Analyticity]\label{analyticity}
Suppose that $f\in C^\infty(\rr)$ and, for any $R>0$, there exists $M>0$ such that
\[
|f^{(n)}(x)|\leq M^{n+1}n!,\qquad \mbox{for\;all}\,\, |x|<R,\,\, n\in\mathbb{N}.
\]
Then any solitary wave solution $\ff\in\z\cap H^\infty(\rr^2)$ of \eqref{main-shrira} is real analytic in $\rr^2$.
\end{theorem}
\begin{proof}
Fix any $(x_0,y_0)\in\rr^2$. To simplify notation, let $P=(x,y)$ and $P_0=(x_0,y_0)$. By Taylor's formula and the smoothness of $\ff$, one has for any $N\in\mathbb{N}$,
\[
\begin{split}
\ff(P+P_0)&=\sum_{|\alpha|\leq N}\frac{P^\alpha}{\alpha!}\partial^\alpha\ff(P_0)+
\sum_{|\alpha|=N+1}\frac{N+1}{\alpha!}\int_0^1(1-t)^NP^\alpha\partial^\alpha\ff(tP+P_0)\;\dd t\\
&=I+II,
\end{split}
\]
where for any $\alpha=(\alpha_1,\alpha_2)\in\N^2$, $$\partial^{\alpha}=\frac{\partial^{\alpha_1+\alpha_2}}{\partial_x^{\alpha_1}\partial_y^{\alpha_2}},$$
represents the derivative operator of order $|\alpha|=\alpha_1+\alpha_2$.
In order to show the Taylor series is absolutely convergent one needs to estimate the second term. By using the regularity of $\ff$ and the Sobolev embedding $H^2(\rrt)\hookrightarrow L^\infty(\rrt)$, one gets, for any $N>2$,
\[
\begin{split}
|II|&\leq \sum_{|\alpha|=N+1} \frac{N+1}{\alpha!}|P|^{|\alpha}\|\partial^\alpha \ff\|_{H^2}\int_0^1(1-t)^N\; \dd t\\
&=\sum_{|\alpha|=N+1} \frac{1}{\alpha!}|P|^{|\alpha}\|\partial^\alpha \ff\|_{H^2}.
\end{split}
\]
We claim that it suffices to show that there are constants $C>0$ and $A>0$ such that, for any $\alpha\in\mathbb{N}^2$,
\begin{equation}\label{partial-est}
\|\partial^\alpha\ff\|_{H^2}\leq CA^{(|\alpha|-1)_+}(|\alpha|-2)_+!,
\end{equation}
where $(\cdot)_+=\max\{\cdot,0\}$. Indeed, assuming \eqref{partial-est}, we deduce
\[
\begin{split}
|II|&\leq C\sum_{|\alpha|=N+1}\frac{1}{\alpha!}|P|^{|\alpha}A^{(|\alpha|-1)_+}(|\alpha|-2)_+!\\
&=CA^N|P|^{N+1}\sum_{|\alpha|=N+1}\frac{(N-1)!}{\alpha!}
\end{split}
\]
Now by using  the elementary inequality (see \cite[Lemma 4.5]{kato-pipolo})
\[
\sum_{|\alpha|=N}\frac{N!}{\alpha!}\leq 2^{N+1},
\]
one gets
\[
|II|\leq \frac{4C|P|}{N(N+1)}(2A|P|)^N.
\]
Thus by taking  a small enough $R$ such that $2AR<1$, we conclude that $II\to0$, as $N\to\infty$, which shows that the Taylor series is absolutely convergent and $\ff$ is real analyticity in a neighborhood of $P_0$.

Now it remains to prove \eqref{partial-est}. If $|\alpha|\leq2$, in view of the inequality $\|\partial^\alpha\ff\|_{H^2}\leq C \|\ff\|_{H^4}$, the proof is direct. For  $|\alpha|>2$, the proof is by induction on $|\alpha|$. Assume the statement is true for all multi-indices $\alpha\in\N^2$ such that $|\alpha|\leq n$. Then, it suffices to show that \eqref{partial-est} holds with $\partial^\alpha\ff$ replaced by $\partial^\alpha\nabla\ff$. First we recall that equation \eqref{shrira-0} is equivalent to
\begin{equation}\label{shrira-h}
\h\ff_x-\Delta\ff+\h f(\ff)_x=0.
\end{equation}
Then by using the regularity of $\ff$, applying the operator $\partial^\alpha$ and taking the inner product in ${H^2(\rrt)}$ with $\partial^\alpha\ff$ in \eqref{shrira-h}, one derives the identity
\begin{equation}
\langle\h\partial^\alpha\ff_x,\partial^\alpha\ff\rangle_{H^2}-
\langle\Delta\partial^\alpha\ff,\partial^\alpha\ff\rangle_{H^2}
=-\langle\h \partial^\alpha f(\ff)_x,\partial^\alpha\ff\rangle_{H^2}.
\end{equation}
Since $\langle\h\partial^\alpha\ff_x,\partial^\alpha\ff\rangle_{H^2}=\|\hd\partial^\alpha\ff\|^2_{H^2}$ and $\langle\Delta\partial^\alpha\ff,\partial^\alpha\ff\rangle_{H^2}=-\|\nabla\partial^\alpha\ff\|_{H^2}^2$,we have
\[\begin{split}
\|\hd\partial^\alpha\ff\|^2_{H^2}+
\|\nabla\partial^\alpha\ff\|_{H^2}^2
&\leq\|\partial^\alpha\h\ff_x\|_{H^2}
\|\partial^\alpha f(\ff)\|_{H^2}\\
&\leq\frac{1}{2}\|D_x^{1/2}\partial^\alpha\ff_x\|_{H^2}^2
+\frac{1}{2}\|\partial^\alpha f(\ff)\|_{H^2}^2.
\end{split}\]
Hence
\begin{equation}\label{a-est}
\|\nabla\partial^\alpha\ff\|_{H^2}\lesssim
\|\partial^\alpha f(\ff)\|_{H^2}.
\end{equation}
Thus, in order to complete the proof one needs to estimate $\|\partial^\alpha f(\ff)\|_{H^2}$.
But, by recalling the estimate \cite[proof of Lemma 4.4]{kato-pipolo}:
\begin{equation}\label{partial-est1}
\|\partial^\alpha f(\ff)\|_{H^2(\rrt)}\leq
CA^{|\alpha|}(|\alpha|-1)!,
\end{equation}
we immediately conclude the proof.
\end{proof}

\begin{remark}
It is to be observed that \eqref{partial-est1} was obtained in \cite{kato-pipolo} when studying analyticity of solitary waves for the KP equation. However, a simple inspection in the proof reveals that it does not depend on the solution itself, but only its smoothness and our assumptions on the nonlinearity $f$.
\end{remark}

\begin{remark}
A similar result of analyticity was obtained, in \cite{lopez}, when $f(u)=u^2$.  Thus Theorem \ref{analyticity} can also be viewed as an extension of that result.
\end{remark}

\subsection{Decay} This subsection is devoted to the study of decay properties of the solitary waves. Our results are inspired in those in \cite{bl} (see also \cite{maris} and \cite{za}).  The difficulty here, once again comes from the fact that the linear part of \eqref{shrira-0} is nonlocal and non-isotropic.

Recall we are assuming $c=1$ and a priori $f$ satisfies  $|f(u)|\leq C|u|^{p-1}, p\in(2,4)$. However, as we will see in the next result,  a further restriction on $p$ must be imposed.

\begin{lemma}\label{35lemma}
Assume that $p\in(p_0,4)$ where $p_0=(3+\sqrt5)/2$. Then any solitary wave $\ff\in\z$ of \eqref{shrira-0} satisfies
\[
\int_\rrt y^2\left(|\hd\ff|^2+|\nabla\ff|^2\right)d xd y<\infty.
\]
\end{lemma}
\begin{proof}
Let $\chi_0\in C_0^\infty(\rr)$ be a function such that $0\leq\chi_0\leq 1$, $\chi(y)=1$ if $|y|\in[0,1]$ and $\chi_0(y)=0$ if $|y|\geq2$. Set $\chi_n(y)=\chi_0(\frac{y^2}{n^2})$, $n\in\mathbb{N}$. Equation \eqref{shrira-0} is equivalent to
\begin{equation}\label{shrira-hh}
\h\ff_x-\Delta\ff+\h f(\ff)_x=0.
\end{equation}
Multiplying \eqref{shrira-hh} by $y^2\chi_n(y)\ff$ and integrating over $\rrt$, we obtain after several integration by parts that
\begin{equation*}
\int_\rrt\h\ff_x\chi_n(y)y^2\ff=
\int_\rrt\chi_n(y)y^2|\hd\ff|^2,
\end{equation*}
\begin{equation*}
-\int_\rrt \ff_{xx}\chi_n(y)y^2\ff=
\int_\rrt\chi_n(y)y^2|\ff_x|^2,
\end{equation*}
and
\begin{equation*}
\begin{split}
-\int_\rrt \ff_{yy}\chi_n(y)y^2\ff=
\int_\rrt\chi_n(y)(y^2\ff_y^2-\ff^2)
-\int_\rrt(2\chi'_n(y)y+\frac{1}{2}\chi''_n(y)y^2)\ff^2.
\end{split}
\end{equation*}
Hence,
\begin{equation}\label{est-0}
\begin{split}
\int_\rrt y^2\chi_n\left(|\hd\ff|^2+|\nabla\ff|^2\right)
&=
\int_\rrt \chi_n\ff^2
+
\int_\rrt(2y\chi_n'+\frac{1}{2}y^2\chi''_n)\ff^2\\
&\quad-
\int_\rrt y^2\chi_n\h\ff f(\ff)_x .
\end{split}\end{equation}
 Let us estimate the last term on the right-hand side of \eqref{est-0}.  By using H\"older's inequality (in the $x$ variable) and the fractional chain rule, we have
\[\begin{split}
\left|\int_{\rr^2} y^2\chi_n\mathscr{H} \ff f(\ff)_x\right|&=
\left|\int_{\rr^2} y^2\chi_n\mathscr{H}\partial_x\ff f(\ff)\right|=\left|\int_{\rr^2} y^2\chi_n D_x^{1/2}\ff D_x^{1/2}f(\ff)\right|\\
&\leq
\int_{\rr}y^2\chi_n\|D_x^{1/2}\ff\|_{L_x^2}\|D_x^{1/2}f(\ff)\|_{L_x^2}dy\\
&\leq
C\int_{\rr}y^2\chi_n\|D_x^{1/2}\ff\|_{L_x^2}\|D_x^{1/2}\ff\|_{L_x^m}\|\ff\|_{L_x^{\ell(p-2)}}^{p-2}dy,
\end{split}
\]
where $m,\ell\neq\infty$ are such that $1/m+1/\ell=1/2$.
Let $\theta=\frac{2(p-2)}{m(p-1)}$, $\lambda=\frac{\ell-2}{\ell(p-1)}$ and take $m$ such that $m>\max\{2,4(p-2)/(p-1)\}$. From the fractional Gagliardo-Nirenberg inequality (see, for instance, \cite{hmoli}), we deduce
\[
\|D_x^{1/2}\ff\|_{L^m_x}\leq C\|\ff_x\|_{L^2_x}^\theta\|\ff\|_{L_x^{2(p-2)}}^{1-\theta},\quad
\|\ff\|_{L_x^{\ell(p-2)}}\leq C\|\ff_x\|_{L_x^2}^\lambda\|\ff\|_{L_x^{2(p-2)}}^{1-\lambda}.
\]
Note that if $p\in[p_0,4)$ then, from Theorem \ref{regularity}, $\ff(\cdot,y)\in L^{2(p-2)}_x$ a.e. $y\in\rr$. In particular note that $2(p-2)\geq 2/(p-1)$ only if $p\geq p_0$ (here is where the restriction on $p$ appears).
Therefore,
\[\begin{split}
\left|\int_{\rr^2} y^2\chi_n\mathscr{H} \ff f(\ff)_x\right|
\leq C\int_{\rr}y^2\chi_n\|D_x^{1/2}\ff\|_{L_x^2}\|\ff_x\|_{L_x^2}\|\ff\|_{L_x^{2(p-2)}}^{p-2}dy.
\end{split}
\]

Let $\epsilon\in(0,1)$ be such that $C\epsilon<1/2$. Since $\ff$ is continuous and tends to zero at infinity, we choose $R>0$ such that $\|\ff(\cdot,y)\|_{L_x^{2(p-2)}}<\epsilon$ for any $|y|>R$. Then, there exists a constant $C_R>0$ such that
\begin{equation}\label{est-00}
\begin{split}
\left|\int_{\rr^2} y^2\chi_n\mathscr{H} \ff f(\ff)_x\right|
&\leq C_R+C\epsilon\int_{|y|>R}y^2\chi_n\|D^{1/2}\ff\|_{L_x^2}\|u_x\|_{L_x^2}dy\\
&\leq C_R+C\epsilon\|y\chi_n^{1/2}D_x^{1/2}\ff\|_{L^2}\|y\chi_n^{1/2}\ff_x\|_{L^2}\\
&\leq C_R+C\epsilon(\|y\chi_n^{1/2}D_x^{1/2}\ff\|_{L^2}^2+\|y\chi_n^{1/2}\ff_x\|_{L^2}^2)\\
&\leq C_R+\frac{1}{2}\left(\|y\chi_n^{1/2}D_x^{1/2}\ff\|_{L^2}^2+\|y\chi_n^{1/2}\ff_x\|_{L^2}^2\right).
\end{split}
\end{equation}
By replacing \eqref{est-00} into \eqref{est-0} we obtain
\begin{equation}\label{est-01}
\int_\rrt y^2\chi_n\left(|\hd\ff|^2+|\nabla\ff|^2\right)
\lesssim
\int_\rrt \chi_n\ff^2
+
\int_\rrt(2y\chi_n'+\frac{1}{2}y^2\chi''_n)\ff^2+C_R.
\end{equation}
The first term of the right-hand side of \eqref{est-01} tends, as $n\to\infty$, to $\|\ff\|_{L^2}^2$ by the Lebesgue theorem. The second term tends to zero by the Lebesgue theorem and the properties of $\chi_n$.
 Therefore
\[
\int_\rrt y^2\chi_n\left(|\hd\ff|^2+|\nabla\ff|^2\right)\dd x\dd y
\]
is uniformly bounded in $n$. By the Fatou lemma, we get our claim.
\end{proof}

In view of Lemma \ref{35lemma} in what follows, otherwise is stated, we assume $p\in(p_0,4)$. However, we believe the restriction on $p$ in Lemma \ref{35lemma} is a technical requirement and once the lemma has been proved for $p\in(2,p_0]$, the results below will also follow.

\begin{lemma}\label{d-kernel}
Let $\ell\geq0$, $\nu>-3/2$ and define $h_\nu$ via its Fourier transform by
\[
\widehat{h_\nu}(\xi,\eta)=\frac{|\xi|^{1+\nu}}{|\xi|+\xi^2+\eta^2}.
\]
Then,
\begin{enumerate}[(i)]
\item $h_\nu\in C^\infty(\rr^2\setminus\{0\})$.
\item $|y|^\ell h_\nu\in L^q(\rrt)$, if $1\leq q<\infty$ and $\frac{\ell}{3}+\frac{1}{q}<\frac{2}{3}\nu+1$ and $\ell+\frac{2}{q}>\nu+1$.
\item $|y|^{2\nu+3} h_\nu\in L^\infty(\rrt)$.
\end{enumerate}
\end{lemma}
\begin{proof}
We observe for any that, $\phi\in \mathcal{S}(\rrt)$ (the Schwartz space)
\[
\begin{split}
\langle h_{\nu},\phi\rangle_{\mathcal{S},\mathcal{S}'}&=\int_{\rrt}\frac{|\xi|^{1+\nu}}{|\xi|+\xi^2+\eta^2}
\int_\rrt\ee^{\ii(x\xi+y\eta)}\phi(x,y)\;\dd x\dd y\;\dd\xi\dd\eta\\
&=\int_0^{+\infty}\int_{\rrt}\int_{\rrt}|\xi|^\nu\ee^{\ii(x\xi+y\eta)}\ee^{-t(1+|\xi|+ \eta^2/|\xi|)}\phi(x,y)\;\dd\xi\dd\eta\;
\dd x\dd y\;\dd t\\
&=\int_0^{+\infty}\frac{\ee^{-t}}{\sqrt{t}}\int_{\rrt}\int_{\rr}|\xi|^{\nu+\frac{1}{2}}\ee^{\ii x\xi}\ee^{-|\xi|(t+y^2/t)}\phi(x,y)\;\dd\xi\;
\dd x\dd y\;\dd t\\
&=2\Gamma(\nu+\frac{3}{2})\int_0^{+\infty}t^{\nu+1}\ee^{-t}\int_{\rrt}
\left(t^2x^2+\left(t^2+y^2\right)^2\right)^{-\frac{2\nu+3}{4}}\\
&\qquad\times\cos\left((\nu+\frac{3}{2})\arctan\left(\frac{t|x|}{t^2+y^2}\right)\right)\phi(x,y)\;\dd x\dd y\;\dd t,
\end{split}
\]
where in the last equality we used that $\xi\mapsto |\xi|^{\nu+\frac{1}{2}}\ee^{-|\xi|(t+y^2/t)}$ is an even function and formula (7) in \cite[page 15]{erd}.
Thus, we deduce that
\begin{equation}\label{kernel}
\begin{split}
h_\nu(x,y)=
2\Gamma(\nu+\frac{3}{2})\int_0^{+\infty}&t^{\nu+1}\ee^{-t}\left(t^2x^2+\left(t^2+y^2\right)^2\right)^{-\frac{2\nu+3}{4}}\\
&\times
\cos\left((\nu+\frac{3}{2})\arctan\left(\frac{t|x|}{t^2+y^2}\right)\right)\;\dd t.
\end{split}
\end{equation}
From the above expression, parts (i) and (iii) are obvious. Let us establish (ii). Indeed,
\begin{equation}\label{calch}
\||y|^\ell h_\nu\|_{L^q}\lesssim \int_0^\infty t^{\nu+1}\ee^{-t}\underbrace{\||y|^{\ell}\left(t^2x^2+\left(t^2+y^2\right)^2\right)^{-\frac{2\nu+3}{4}}\|_{L^q}}_{A}\; \dd t.
\end{equation}
But,
\begin{equation}\label{innerint}
\begin{split}
A^q&=\int_{\rr}|y|^{q\ell}\left[\int_{\rr}\left(t^2x^2+\left(t^2+y^2\right)^2\right)^{-\frac{2\nu+3}{4}q}\dd x\right]\dd y\\
&= \int_{\rr}|y|^{q\ell}\left[ t^{-\frac{2\nu+3}{2}q}\int_{\rr}\left(x^2+\left(t+\frac{y^2}{t}\right)^2\right)^{-\frac{2\nu+3}{4}q}\dd x\right]\dd y\\
&= \int_{\rr}|y|^{q\ell}\left[ t^{-\frac{2\nu+3}{2}q} \left(t+\frac{y^2}{t}\right)^{1-\frac{2\nu+3}{2}q}
\int_{\rr}\left(z^2+1\right)^{-\frac{2\nu+3}{4}q}\dd z\right]\dd y,
\end{split}
\end{equation}
where we used a change of variable. Since $\nu\geq-3/2$ we have
$$
\frac{1}{q}<\frac{2}{3}\nu+1-\frac{\ell}{3}\leq \frac{2}{3}\nu+1<\nu+\frac{3}{2}
$$
and the inner integral in \eqref{innerint} is finite. Thus,
\begin{equation}\label{innerint1}
\begin{split}	
A^q&=C t^{-\frac{2\nu+3}{2}q} \int_{\rr}|y|^{q\ell} \left(t+\frac{y^2}{t}\right)^{1-\frac{2\nu+3}{2}q}
	\dd y\\
&=Ct^{-1} \int_{\rr}|y|^{q\ell} \left(t^2+y^2\right)^{1-\frac{2\nu+3}{2}q}\dd y\\
&=Ct^{q\ell+2-(2\nu+3)q}\int_{\rr}|z|^{q\ell} \left(1+z^2\right)^{1-\frac{2\nu+3}{2}q}\dd y.
\end{split}
\end{equation}
Since $\frac{\ell}{3}+\frac{1}{q}<\frac{2}{3}\nu+1$ this last integral is also finite. Consequently,
$$
\||y|^\ell h_\nu\|_{L^q}\lesssim \int_0^\infty t^{\nu+1}\ee^{-t}t^{\ell+\frac{2}{q}-(2\nu+3)}\dd t.
$$
The assumption $\ell+\frac{2}{q}>\nu+1$ now implies that this last integral is finite and the proof of the lemma is completed.
\end{proof}

\begin{lemma}\label{decay-1}
	Assume $f\in C^1$ and  $p\geq p_0=(3+\sqrt{5})/2$. Let  $\ff\in\z$ be any solitary wave  of \eqref{shrira-0}. Then $|y|\ff\in L^q(\rrt)$, for all $3/2< q\leq\infty.$
\end{lemma}
\begin{proof}
First we show $y\ff\in L^\infty(\rr^2)$. Choose $\beta\in(0,3/4)$ and $q_1,q_2>2$ satisfying
\begin{equation}\label{roleq}
	\frac{1}{2}=\frac{1}{q_1}+\frac{1}{q_2},\qquad
	\frac{3}{2q_1}<\beta<\frac{2}{q_1}\quad\mbox{and}\quad (p-2)q_2>1,
\end{equation}
	where $h_{\beta-1}$ is as in Lemma \ref{d-kernel}.
	Then it follows from
	\[
	|y\ff|\lesssim|yh_{\beta-1}|\ast|D^{1-\beta}_xf(\ff)|+|h_{\beta-1}|\ast|yD^{1-\beta}_xf(\ff)|,
	\]
	 Young's inequality and the fractional chain rule that
	\[\begin{split}
	\|y\ff\|_{L^\infty}&\lesssim \|yh_{\beta-1}\|_{L^{q_1}}\|D^{1-\beta}_x\ff\|_{L^2}\|\ff^{p-2}\|_{L^{q_2}}\\
	&\quad+
	\|h_{\beta-1}\|_{L^{q_1}}\|yD^{1-\beta}_x\ff\|_{L^2}\|\ff^{p-2}\|_{L^{q_2}}.
	\end{split}\]
	From \eqref{roleq}, Lemma \ref{d-kernel}, Theorem \ref{regularity}, and the fact that  $\ff\in\z$, the right-hand side of the above inequality is finite if $\|yD^{1-\beta}_x\ff\|_{L^2}<+\infty$.
	
	We state that  $\|yD^{1-\beta}_x\ff\|_{L^2}<+\infty$. Indeed,  if $\ff$ satisfies \eqref{shrira-0}, then
	\[
	D^{1-\beta}_x\ff=h_{\tilde\beta}\ast \hd f(\ff),
	\]
	where $\tilde\beta=1/2-\beta$.  Now we choose $\widetilde{q}_1,\widetilde{q}_2,r_1$ and $r_2$ such that
	 \begin{equation}\label{roleq1}
	  1=\frac{1}{\widetilde{q}_1}+\frac{1}{\widetilde{q}_2},\quad \frac{2}{\widetilde{q}_1}>\tilde\beta>\frac{3}{2\widetilde{q}_1}-1, \quad \widetilde{q}_1>1, \quad  (p-2)\widetilde{q}_2>1
	 \end{equation}
	 and
	 \begin{equation}\label{roleq2}
	  1=\frac{1}{r_1}+\frac{1}{r_2}, \quad \frac{1}{2}(\frac{3}{r_1}-1)<1+\tilde\beta<\frac{2}{r_1}, \quad  r_1>1, \quad (p-2)r_2>1.
	 \end{equation}
Since
	\[
|yD^{1-\beta}\ff|\lesssim|yh_{\tilde\beta}|\ast|\hd f(\ff)|+|h_{\tilde\beta}|\ast|y\hd f(\ff)|,
\] we have from the fractional chain rule and Young's inequality
	\[
	\begin{split}
	\|yD^{1-\beta}\ff\|_{L^2}&\lesssim \|yh_{\tilde\beta}\|_{L^{q_1}}\|\hd\ff\|_{L^2}\|\ff^{p-2}\|_{L^{q_2}}\\
	&\quad+
	\|h_{\tilde\beta}\|_{L^{r_1}}\|y\hd\ff\|_{L^2}\|\ff^{p-2}\|_{L^{r_2}}<+\infty.
	\end{split}\]
Note that since $p\geq p_0$, Lemma \ref{35lemma} implies that $\|y\hd\ff\|_{L^2}<\infty$. Hence,  \eqref{roleq1}, \eqref{roleq2}, Lemma \ref{d-kernel}, Theorem \ref{regularity}, and the fact that  $\ff\in\z$, implies that the right-hand side of the above inequality is finite.

Next we prove that $y\ff\in L^q(\rr^2)$, $q>3/2$. Because
\[
|y\ff|\lesssim|yh_0|\ast|f(\ff)|+|h_0|\ast|y f(\ff)|,
\]
by choosing $q_1\in(1,2)$, $r_1>3/2$ and $r_2(p-1),q_2(p-2)>1$ satisfying $1+\frac{1}{q}=\frac{1}{q_1}+\frac{1}{q_2}=\frac{1}{r_1}+\frac{1}{r_2}$, we get from $y\ff\in L^\infty$ that
\[\begin{split}
\|y\ff\|_{L^q}&\lesssim \|yh_0\|_{L^{r_1}}\|f(\ff)\|_{L^{r_2}}
+
\|h_0\|_{L^{q_1}}\|yf(\ff)\|_{L^{q_2}}\\
&\lesssim
\|\ff\|_{L^{r_2(p-1)}}^{p-1}+\|y\ff\|_{L^\infty}\|\ff\|_{L^{q_2(p-2)}}^{p-2}<+\infty,
\end{split}\]
where we used Lemma \ref{d-kernel} and Theorem \ref{regularity} again.
	This completes the proof.
\end{proof}

In view of Lemma \ref{decay-1}, otherwise stated, we assume that $f\in C^1$.
As an immediate consequence we deduce.

\begin{corollary}
Let $\ff\in\z$ be any solitary wave  of \eqref{shrira-0} and $\theta\in[0,1]$. Then $|y|^\theta\ff\in L^q(\rrt)$, for all $3/2< q\leq\infty$.
\end{corollary}
\begin{proof}
It suffices to note that
\[
\begin{split}
\int_{\rrt}|y|^{\theta q}|\ff|^q\dd x\dd y&=\int_{\rr}\left(\int_{|y|\leq1}|y|^{\theta q}|\ff|^q\dd y \right)\dd x+\int_{\rr}\left(\int_{|y|\geq1}|y|^{\theta q}|\ff|^q\dd y \right)\dd x\\
&\leq \int_{\rrt}|\ff|^q\dd x\dd y+\int_{\rrt}|y|^{q}|\ff|^q\dd x\dd y
\end{split}
\]
and apply Theorem \ref{regularity} and Lemma \ref{decay-1}.
\end{proof}

Next, we observe that equation \eqref{shrira-0} may be written in the equivalent form
\begin{equation}\label{conv-form}
\ff=k\ast f(\ff),
\end{equation}
where $k=h_0=\left(\frac{|\xi|}{|\xi|+\xi^2+\eta^2}\right)^\vee$ was defined in Lemma \ref{d-kernel}. We will use the properties of the kernel $k$ to get some decay estimates for the solution $\ff$ of \eqref{shrira-0}. As an immediate consequence of Lemma \ref{d-kernel} we have the following.

\begin{lemma}\label{kernel-properties}
Let $\ell\in[0,3)$. Assume that  $1\leq q<\infty$ satisfies $\frac{1}{q}+\frac{1}{3}\ell<1$ and $1<\ell+\frac{2}{q}$. Then,
\begin{enumerate}[(i)]
\item $k\in C^\infty(\rr^2\setminus\{0\})$;
\item $|y|^\ell k\in L^q(\rrt)$;
\item $|y|^{3} k\in L^\infty(\rrt)$.
\end{enumerate}
\end{lemma}

Concerning decay and integrability with respect to a power of $x$ we have the following result.
\begin{lemma}\label{kernel-properties1}
	Let $\ell\in[0,3/2)$. Assume that  $1\leq q<\infty$ satisfies $\frac{1}{q}+\frac{2}{3}\ell<1$ and $1<\ell+\frac{2}{q}$.
	\begin{enumerate}[(i)]
		\item $|x|^\ell k\in L^q(\rrt)$;
		\item $|x|^{3/2} k\in L^\infty(\rrt)$.
	\end{enumerate}
\end{lemma}
\begin{proof}
The proof is very similar to that of Lemma \ref{d-kernel}; so we omit the details. We only point out that a power $|z|^{q\ell}$ appears in the inner integral \eqref{innerint}. Thus, a condition for integrability is $\frac{3}{2}q-q\ell>1$; but this is true because $\frac{1}{q}+\frac{2}{3}\ell<1$ and $\ell<3/2$.
\end{proof}

The next step is to show that solutions of \eqref{shrira-0} decay to zero at infinity at the same rate as the kernel $k$.

\begin{theorem}[Spatial decay in the $y$ variable]\label{inf-decay-y}
Any solitary wave $\ff\in\z$ of \eqref{shrira-0} satisfies
\begin{itemize}
	\item[(i)] $y^3\ff\in L^\infty(\rrt)$; and
	\item[(ii)] $|y|^\kappa\ff\in L^\infty(\rrt)$, $0\leq \kappa\leq3$.
\end{itemize}

\end{theorem}
\begin{proof}
It suffices to prove (i), because (ii) follows immediately from (i).
First, by using \eqref{conv-form}, we recall the trivial inequality
\begin{equation}\label{conv-1}
|y|^{\ell}|\ff|\lesssim|y| ^{\ell}|k|\ast|f(\ff)|+|k|\ast||y|^{\ell}|f(\ff)|,
\end{equation}
which holds for any $\ell\geq0$. Let $\gamma_1=p-1$.\\
\noindent {\bf Claim 1.}
 $|y|^{\gamma_1}\ff\in L^r(\rrt)$, for any $1\leq r\leq\infty$.

 Indeed, choose $r_1,r_2,q_1,q_2$ such that  $1+\frac{1}{r}=\frac{1}{r_1}+\frac{1}{r_2}=\frac{1}{q_1}+\frac{1}{q_2}$, $r_1\in(1,2)$, $q_2\gamma_1>1$, $r_2\gamma_1>\frac{3}{2}$ and $q_1>\frac{3}{3-\gamma_1}$.
From \eqref{conv-1} with $\ell=\gamma_1$ and the Young inequality it follows that
 \[
 \||y|^{\gamma_1}\ff\|_{L^r}
 \lesssim
 \||y| ^{\gamma_1}k\|_{L^{q_1}}\|\ff\|_{L^{q_2\gamma_1}}^{\gamma_1}
 +
 \|k\|_{L^{r_1}}\|y\ff\|_{L^{r_2\gamma_1}}^{\gamma_1}.
 \]
Thanks to Lemma \ref{kernel-properties}, Lemma \ref{decay-1}, and Theorem \ref{regularity}, the right-hand side of the above inequality is finite and the claim is established.\\

Next we define  $\gamma_2=\min\{3,(p-1)^2\}$ and divide the proof into two cases.\\

\noindent {\bf Case 1.}  $\gamma_2=(p-1)^2$.

 Here, we observe the following

 \noindent {\bf Claim 2.}  $|y|^{\gamma_2}\ff\in L^r(\rr^2)$,
 for any $\frac{3}{\gamma_1(3-\gamma_1)}<r\leq\infty$.

In fact,  by using \eqref{conv-1} with $\ell=\gamma_2$ and an argument similar to that in  Claim 1, we have
\[
\||y|^{\gamma_2}\ff\|_{L^r}
\lesssim
\||y| ^{\gamma_2}k\|_{L^{q_1}}\|\ff\|_{L^{q_2\gamma_1}}^{\gamma_1}
+
\|k\|_{L^{r_1}}\||y|^{\gamma_1}\ff\|_{L^{r_2\gamma_1}}^{\gamma_1}<+\infty,
\]
provided $1+\frac{1}{r}=\frac{1}{r_1}+\frac{1}{r_2}=\frac{1}{q_1}+\frac{1}{q_2}$, $r_1\in(1,2)$, $q_2\gamma_1>1$, $r_2\gamma_2>\frac{3}{2}$ and $q_1>\frac{3}{3-\gamma_2}$. Note we used Claim 1 for the last term.   From our choices, $\frac{1}{q_2}<\gamma_1$ and $\frac{1}{q_1}<\frac{3-\gamma_2}{3}$, which implies $\frac{1}{q_1}+\frac{1}{q_2}<1+\frac{\gamma_1(3-\gamma_1)}{3}$. This forces the restriction $r>\frac{3}{\gamma_1(3-\gamma_1)}$ and shows our claim.

Now, since $p\geq p_0=(3+\sqrt{5})/2$, we deduce that $\gamma_1\gamma_2>3>\gamma_2$. Thus,
\begin{equation}\label{inflem1}
\|y^3\ff\|_{L^{\infty}}\lesssim
\|y^3k\|_{L^{\infty}}\|\ff\|_{L^{\gamma_1}}^{\gamma_1}+
\|k\|_{L^{a}}\||y|^{\frac{3}{p-1}}\ff\|_{L^{b\gamma_1}}^{\gamma_1},
\end{equation}
where $1=\frac{1}{a}+\frac{1}{b}$. From Theorem \ref{regularity} and Lemma \ref{kernel-properties} the first term on the right-hand side of \eqref{inflem1} is finite. Also, by choosing $a\in(0,1)$ and $b$ satisfying $b\gamma_2>\frac{3}{3-\gamma_1}$, we obtain $k\in L^a(\rrt)$ and
\begin{equation}\label{inflem2}
\||y|^{\frac{3}{p-1}}\ff\|_{L^{b\gamma_1}}^{\gamma_1}\lesssim
\|\ff\|_{L^{b\gamma_1}}^{\gamma_1}+\||y|^{\gamma_2}\ff\|_{L^{b\gamma_1}}^{\gamma_1},
\end{equation}
where we used $\frac{3}{p-1}=\frac{3}{\gamma_1}<\gamma_2$. The right-hand side of \eqref{inflem2} is finite thanks to Theorem \ref{regularity} and Claim 2. This proves that the right-hand side of \eqref{inflem1} is finite and concludes the proof of the theorem in this case.\\

\noindent {\bf Case 2.}  $\gamma_2=3$.

Here, if $\frac{1}{a}+\frac{1}{b}=1$, we write
\[
\begin{split}
\|y^3\ff\|_{L^{\infty}}&\lesssim \|y^3k\|_{L^{\infty}}\|\ff\|_{L^{\gamma_1}}^{\gamma_1}+
\|k\|_{L^{a}}\||y|^{3}f(\ff)\|_{L^{b}}\\
&\lesssim \|y^3k\|_{L^{\infty}}\|\ff\|_{L^{\gamma_1}}^{\gamma_1}+
\|k\|_{L^{a}}\left(\|\ff\|_{L^{b\gamma_1}}^{\gamma_1}\||y|^{\gamma_1}\ff\|_{L^{b\gamma_1}}^{\gamma_1} \right)
\end{split}
\]
where we used that $3<(p-1)^2=\gamma_1^2$. By choosing $a\in(0,1)$, $b\gamma_1>1$, using Claim 1, and arguing as in Case 1, we complete the proof of the theorem.
\end{proof}

Interest is now turned to the decay with respect to the variable $x$. Let us start with the following result.

\begin{lemma}\label{decay-x}
Let $q_0=2(p-1)$. Then, for any $q\in(q_0,\infty)$ and $\ell\geq0$ satisfying $\ell q<1/2$, we have $|x|^{\ell}\ff\in L^q(\rrt)$.
\end{lemma}

Before proving Lemma \ref{decay-x} we recall the following.

\begin{lemma}\label{integal-estimates}
Let $j\in\mathbb{N}$. Suppose also that $\ell$ and $m$ are two constants satisfying $0<\ell<m-j$.
Then there exists $C>0$, depending only on $\ell$ and $m$,
such that for all $\epsilon\in(0,1]$,  we have
\begin{equation}\label{est-1}
\int_{\rr^j}\frac{|a|^\ell}{(1+\epsilon|a|)^m(1+|b-a|)^m}\;\dd a
\leq\frac{C\;|b|^\ell}{(1+\epsilon|b|)^m},\qquad\forall\;b\in\rr^j,\;|b|\geq1,
\end{equation}
and
\begin{equation}\label{est-2}
\int_{\rr^j}\frac{\dd a}{(1+\epsilon|a|)^m(1+|b-a|)^m}
\leq\frac{C}{(1+\epsilon|b|)^m},\qquad\forall\;b\in\rr^j.
\end{equation}
\end{lemma}
\begin{proof}
The proof is quite elementary and it is essentially the same as that
of Lemma 3.1.1 in \cite{bl} (see also \cite{jpha}).
\end{proof}

\begin{proof}[Proof of Lemma \ref{decay-x}]
Fix $r\in(1,2)$ to be chosen later  and take $s_1\in(\frac{1}{r'},\frac{3}{2r'})$ and $s_2\in(\frac{1}{r'},3(p-2))$, where $r'$ is the H\"older conjugate of $r$.

We first claim that, for any  $\ell\in[0,s_1-\frac{1}{r'})$, we have  $|x|^\ell\langle y\rangle^{-s_2}\ff\in L^{r'}(\rrt)$. Indeed,  for $0<\eps\ll1$, define $g_\eps$ by
\[
g_\eps(x,y)=\frac{|x|^\ell}{{\langle x\rangle}^{s_1}_\eps{\langle y\rangle}^{s_2}}\ff(x,y),
\]
where $\langle y\rangle=1+|y|$ and $\langle x\rangle_\eps=1+\eps|x|$. Since $\ff\in L^\infty(\rrt)$, from the choices of $\ell$ and $s_j,j=1,2$, it  is easy to see that $g_\eps\in L^{r'}(\rrt)$. Now, given  any $\delta > 0$, there exists a constant $N>1$ (depending on $\delta$) such that
\begin{equation}\label{yddecay}
\langle y\rangle^{\frac{s_2}{p-2}}|\ff|<\delta, \qquad \mbox{for}\;\; |x|>N.
\end{equation}
To see this, choose a number $a\in(0,1)$ satisfying $\frac{s_2}{a(p-2)}<3$ (this is possible because $s_2<3(p-2)$). Then,
\[
\begin{split}
\langle y\rangle^{\frac{s_2}{p-2}}|\ff|\lesssim |\ff|+|y|^{\frac{s_2}{p-2}}|\ff|^a|\ff|^{1-a}\leq \||y|^{\frac{s_2}{ap-2}}\ff\|^a_{L^\infty}|\ff|^{1-a}\lesssim |\ff|^{1-a},
\end{split}
\]
where we used Theorem \ref{inf-decay-y}. Since $\ff$ goes to zero at infinity (see Theorem \ref{regularity} and Remark \ref{vanizero}), this last inequality implies \eqref{yddecay}.

Now we decompose $$\rr=\{|x|>N\}\cup\{|x|\leq N\}=:I_1\cup I_2.$$ Then, by using equation \eqref{conv-form}, H\"{o}lder's inequality  and Lemma \ref{kernel-properties1}, we get
\[\begin{split}
\int_{I_1\times\rr}&|g_\eps(x,y)|^{r'}\dd x\dd y\\
&\leq\int_{I_1\times\rr}|g_\eps(x,y)|^{r'-1}\frac{|x|^\ell}{\langle x\rangle^{s_1}_\eps\langle y\rangle^{s_2}}k\ast f(\ff)\dd x\dd y\\
&\leq \int_{I_1\times\rr}|g_\eps(x,y)|^{r'-1}\frac{|x|^\ell}{\langle x\rangle^{s_1}_\eps\langle y\rangle^{s_2}}
\|k\langle x\rangle^{s_1}\|_{L^r(\rrt)}\|\langle x\rangle^{-s_1}\ast f(\ff)\|_{L^{r'}(\rrt)}\dd x\dd y\\
&\lesssim\|g_\eps\|_{L^{r'}(I_1\times\rr)}^{r'-1}\left(\int_{I_1\times\rr}\frac{|x|^{\ell r'}}{\langle x\rangle^{s_1r'}_\eps\langle y\rangle^{s_2r'}}\|\langle x\rangle^{-s_1}\ast f(\ff)\|_{L^{r'}(\rrt)}^{r'}\dd x\dd y\right)^{1/{r'}}.
\end{split}\]
Since $g_\epsilon\in L^{r'}(\rrt)$, we can divide both sides of the above inequality by $\|g_\eps\|_{L^{r'}}^{r'-1}$ to obtain
$$
\int_{I_1\times\rr}|g_\eps(x,y)|^{r'}\dd x\dd y\lesssim \int_{I_1\times\rr}\frac{|x|^{\ell r'}}{\langle x\rangle^{s_1r'}_\eps\langle y\rangle^{s_2r'}}\|\langle x\rangle^{-s_1}\ast f(\ff)\|_{L^{r'}(\rrt)}^{r'}\dd x\dd y.
$$
By using the definition of convolution, Fubini's theorem and Lemma \ref{integal-estimates},
 \[\begin{split}
\int_{I_1\times\rr}&|g_\eps(x,y)|^{r'}\dd x\dd y\\
&\lesssim
\int_\rrt |f(\ff)(x',y')|^{r'}
\int_{I_1\times\rr}\frac{|x|^{\ell r'}}{\langle x\rangle^{s_1r'}_\eps\langle y\rangle^{s_2r'}\langle x-x'\rangle^{s_1r'}}
\dd x\dd y\;\dd x'\dd y'\\
&\lesssim \int_{I_1\times\rr}|f(\ff)(x',y')|^{r'}\frac{|x'|^{\ell r'}}{\langle x'\rangle^{s_1r'}_\eps}\dd x'\dd y'\\
&\qquad+
\int_{I_2\times\rr}|f(\ff)(x',y')|^{r'}
\int_{I_1\times\rr}\frac{|x|^{\ell r'}}{\langle x\rangle^{s_1r'}_\eps\langle y\rangle^{s_2r'}\langle x-x'\rangle^{s_1r'}}
\dd x\dd y\;\dd x'\dd y'\\
&\lesssim\delta^{r'(p-2)}
\int_{I_1\times\rr}|\ff(x',y')|^{r'}\frac{|x'|^{\ell r'}}{\langle x'\rangle^{s_1 r'}_\eps\langle y'\rangle^{s_2r'}}
\dd x'\dd y'\\
&\qquad+
\int_{I_2\times\rr}|\ff(x',y')|^{r'(p-1)}
\int_{\rrt}\frac{(|x|+2N)^{\ell r'}}{\langle x\rangle^{s_1 r'}\langle y\rangle^{s_2 r'}}\dd x\dd y\;\dd x'\dd y'\\
&\lesssim
\delta^{r'(p-2)}\int_{I_1\times\rr}|g_\eps(x',y')|^{r'}\dd x'\dd y'+
\int_{I_2\times\rr}|\ff(x',y')|^{r'(p-1)}\dd x'\dd y'.
\end{split} \]
By choosing $\delta$ sufficiently small, we deduce that
\[
\int_{I_1\times\rr}|g_\eps(x,y)|^{r'}\dd x\dd y
\leq C
\int_{I_2\times\rr}|\ff(x',y')|^{r'(p-1)}\dd x'\dd y'.
\]
Since the constant $C$ appearing in the right-hand side of the preceding estimate
is independent of $\eps$, an application of Fatou's lemma gives
\begin{equation}\label{gepi1}
\int_{I_1\times\rr}\frac{|x|^{\ell r'}}{\langle y\rangle^{s_2r'}}|\ff(x,y)|^{r'}\dd x\dd y
\lesssim 1.
\end{equation}
On the other hand, clearly
\begin{equation}\label{gepi2}
\int_{I_2\times\rr}\frac{|x|^{\ell r'}}{\langle y\rangle^{s_2r'}}|\ff(x,y)|^{r'}\dd x\dd y
\lesssim \|\ff\|_{L^{r'}(\rrt)}^{r'}.
\end{equation}
A combination of \eqref{gepi1} and \eqref{gepi2} then establishes that
 $|x|^\ell\langle y\rangle^{-s_2}\ff\in L^{r'}(\rrt)$ for any $r'\in(2,\infty)$, $\ell\in[0,s_1-1/r')$ and $s_2\in(\frac{1}{r'},3(p-2))$, which is precisely our claim.

In order to complete the proof of the lemma, we fix $q_0=2(p-1)$ and take  $q\in(q_0,\infty)$. Let $\nu=\frac{q}{p-1}$  and note that $2<\nu<q$. Now, for any $\ell\geq0$ and $s>0$, we infer that
\begin{equation}\label{test-eq}
\||x|^{\ell}\ff\|_{L^{q}(\rrt)}\leq
\||x|^{\frac{q\ell}{\nu}}\langle y\rangle^{-\frac{s}{\nu}}\ff\|_{L^{\nu}(\rrt)}^{\frac{\nu}{q}}
\|\langle y\rangle^{\frac{s}{q-\nu}}\ff\|_{L^\infty(\rrt)}^{\frac{q-\nu}{q}}.
\end{equation}
We claim that by choosing  $s\in(1,3\nu(p-2))$, the right-hand side of \eqref{test-eq} is finite. Indeed, since the function $t\mapsto\frac{t}{p-1}-\frac{1}{q}$ tends to 0, as $t\to \frac{p-1}{q}$, and to $\frac{1}{2q}$ as $t\to\frac{3(p-1)}{2q}$, we can find a number $s_1\in\left(\frac{p-1}{q},\frac{3(p-1)}{2q}\right)$ such that
$$
0\leq \ell<\frac{s_1}{p-1}-\frac{1}{q}.
$$
With this inequality in hand, all assumption in our claim above hold and the first term in \eqref{test-eq} is finite. The second one also finite thanks to  Theorem \ref{inf-decay-y}. Note that $\frac{s}{q-\nu}<3$ in view of our choices of $\nu$ and $s$.
This completes the proof.
\end{proof}

\begin{theorem}[Spatial decay in the $x$ variable]
Any nontrivial solitary wave $\ff\in\z$ of \eqref{shrira-0} satisfies $|x|^{3/2}\ff\in L^\infty(\rrt)$.
\end{theorem}
\begin{proof}
	The proof is analogous to   that of Theorem \ref{inf-decay-y}. We divide it into several steps.
	
\noindent {\bf Step 1.}  First we note from \eqref{conv-1} that $|x|^\ell \ff\in L^{q}(\rrt)$ for any $2<q\leq \infty$ and $0\leq2\ell<\frac{1}{2}+\frac{1}{q}$. In fact, by  choosing $r_1,r_2>1$, $q_1\in(1,2)$ and $q_2\in(2,\infty)$ such that  $1+\frac{1}{q}=\frac{1}{r_1}+\frac{1}{r_2}=\frac{1}{q_1}+\frac{1}{q_2}$, $2\ell q_2<1$ and $\frac{1}{r_1}+\frac{2\ell}{3}<1<\ell+\frac{2}{r_1}$, as in \eqref{conv-1}, we get from the Young inequality,
	\[\begin{split}
	\||x|^\ell\ff\|_{L^q}&\lesssim \||x|^\ell k\|_{L^{r_1}}\|\ff\|_{L^{r_2(p-1)}}^{p-1}
	+
	\|k\|_{L^{q_1}}\||x|^\frac{\ell}{p-1} \ff\|_{L^{q_2(p-1)}}^{p-1}<+\infty,
	\end{split}\]
where we used Theorem \ref{regularity}, Lemma \ref{kernel-properties1}, and Lemma \ref{decay-x}.
 The restrictions on $q$ and $\ell$  come  from
	\[
1+\frac{1}{2}>	1+\frac{1}{q}=\frac{1}{q_1}+\frac{1}{q_2}>\frac{1}{2}+2\ell,
	\quad
	1+\frac{1}{q}=\frac{1}{q_1}+\frac{1}{q_2}<1+\frac{1}{2}.
	\]
	In particular, $|x|^\ell\ff\in L^\infty(\rrt)$, if $0\leq\ell<1/4$.\\

\noindent {\bf Step 2.} We now show that $|x|^\ell \ff\in L^{q}(\rrt)$ for any $\max\{1,\frac{2}{p-1}\}=:\overline{q}<q\leq \infty$ and $0\leq2\ell<\frac{p}{2}+\frac{1}{q}$. In fact, by choosing $r_1,r_2>1$, as in Step 1 and $q_1\in(1,2)$ and $q_2\in(\overline{q},\infty)$  such that  $1+\frac{1}{q}=\frac{1}{q_1}+\frac{1}{q_2}$ and $\frac{2\ell}{p-1}<\frac{1}{2}+\frac{1}{q_2(p-1)}$, we deduce
	\[\begin{split}
	\||x|^\ell\ff\|_{L^q}\lesssim \||x|^\ell k\|_{L^{r_1}}\|\ff\|_{L^{r_2(p-1)}}^{p-1}
	+
	\|k\|_{L^{q_1}}\||x|^\frac{\ell}{p-1} \ff\|_{L^{q_2(p-1)}}^{p-1}<+\infty,
	\end{split}\]
	where now to see that the  last term in the above inequality is finite we used the result in Step 1. 	The restrictions on $q$ and $\ell$ follows as in Step 1.
	In particular, $|x|^\ell\ff\in L^\infty(\rrt)$, if $0\leq\ell<p/4$.\\
	
\noindent {\bf Step 3.} We claim that if $p$ satisfies $p(p-1)>5$ then $|x|^{3/2}\ff\in L^\infty$ and the proof of the theorem is completed in this case. Indeed, by choosing  $q_1\in(1,2)$ and $q_2>2$  such that  $1=\frac{1}{q_1}+\frac{1}{q_2}$, we write
	\[
	\||x|^{3/2}\ff\|_{L^\infty}\lesssim \||x|^{3/2} k\|_{L^{\infty}}\|\ff\|_{L^{p-1}}^{p-1}
	+
	\|k\|_{L^{q_1}}\||x|^\frac{3}{2(p-1)} \ff\|_{L^{q_2(p-1)}}^{p-1}<+\infty
\]
The last term in the above inequality is finite in view of Step 2. We point out that conditions on $q$ and $\ell$ in Step 2, is equivalent to $0\leq3<\frac{p(p-1)}{2}+\frac{1}{q_2}$, which holds because
$$
3=\frac{5}{2}+\frac{1}{2}<\frac{p(p-1)}{2}+\frac{1}{q_2}.
$$
This establishes Step 3.\\

Assume from now on that $p$ satisfies $p_0\leq p\leq p_1$, where $p_1$ is the positive root of $p(p-1)=5$.\\

\noindent {\bf Step 4.} We show that $|x|^\ell \ff\in L^{q}(\rrt)$ for any $1<q\leq \infty$ and $0\leq2\ell<\frac{1}{2}+\frac{p(p-1)}{2}+\frac{1}{q}$. Indeed, in order to apply the results in Step 2, we  choose $r_1,r_2>1$, as in Step 1 and $q_1\in(1,2)$ and $q_2\in(1,\infty)$  such that  $1+\frac{1}{q}=\frac{1}{q_1}+\frac{1}{q_2}$ and $\frac{2\ell}{p-1}<\frac{p}{2}+\frac{1}{q_2(p-1)}$. Consequently,
	\[\begin{split}
	\||x|^\ell\ff\|_{L^q}&\lesssim \||x|^\ell k\|_{L^{r_1}}\|\ff\|_{L^{r_2(p-1)}}^{p-1}
	+
	\|k\|_{L^{q_1}}\||x|^\frac{\ell}{p-1} \ff\|_{L^{q_2(p-1)}}^{p-1}<+\infty.
	\end{split}\]
	In particular, $|x|^\ell\ff\in L^\infty(\rrt)$, if $0\leq\ell<\frac{1}{4}+\frac{p(p-1)}{4}=:\ell_0$. Note that $\ell_0<3/2$, which is expected at this stage.\\

\noindent {\bf Step 5.} We finally show that $|x|^{3/2}\ff\in L^\infty(\rrt)$ if $p\in[p_0,p_1]$. In fact, choosing  $q_1\in(1,2)$ and $q_2>2$  satisfying $1=\frac{1}{q_1}+\frac{1}{q_2}$, we get
	\[\begin{split}
	\||x|^{3/2}\ff\|_{L^\infty}&\lesssim \||x|^{3/2} k\|_{L^{\infty}}\|\ff\|_{L^{p-1}}^{p-1}
	+
	\|k\|_{L^{q_1}}\||x|^\frac{3}{2(p-1)} \ff\|_{L^{q_2(p-1)}}^{p-1}<+\infty.
	\end{split}\]
To use Step 4 in order to see that last term is finite, we need to check that $ 3<\frac{p-1}{2}+\frac{p(p-2)^2}{2}+\frac{1}{q_2}$. But note that such a inequality holds trivially if we replace $p$ by $p_0$. Thus the result follows because $p\geq p_0$.

The proof of the theorem is thus completed.
\end{proof}

\begin{remark}
It is worth noting that the solitary wave solution $\ff\in\z$ cannot belong to $L^1(\rrt)$, since $\hat k$ is not continuous at the origin (see \eqref{conv-form}).
\end{remark}

We finish this section with an additional decay property.

\begin{theorem}
	Any nontrivial solitary wave of \eqref{shrira-0} satisfies $\ff\in L_y^rL_x^q(\rrt)\cap L_x^q L_y^r(\rrt)$ for all $1\leq q,r\leq\infty$ satisfying
	\begin{equation}\label{qr-conditions}
	\frac{1}{r}+\frac{1}{q}>1\quad\mbox{and}\quad 	\frac{1}{r}+\frac{2}{q}<3.
	\end{equation}
	In particular $\ff\in L^q_yL^1_x(\rrt)\cap L^1_xL^q_y(\rrt)\cap L^q_xL^1_y(\rrt)\cap L^1_yL^q_x(\rrt)$ for any $1<q\leq\infty$.
\end{theorem}
\begin{proof}
	The proof is deduced from the fact $k\in  L_y^rL_x^q(\rrt)\cap L_x^q L_y^r(\rrt)$ under conditions \eqref{qr-conditions}.
\end{proof}

\section{Appendix} \label{sec4}

An important question concerning traveling-wave solutions one can ask is about their positivity. In this short appendix we verify that under suitable vanishing conditions at infinity, positive solitary waves do not exist. 

\begin{proposition}[Nonexistence of positive solitary waves]
Suppose that $f$ does not change the sign. Then there is no positive solitary wave solution $\ff$ of \eqref{shrira-0} satisfying
\begin{gather}
\ff\to0,\qquad\mbox{as}\quad |(x,y)|\to+\infty,\label{1-c}\\
\h\ff_x\to0,\qquad\mbox{as}\quad |x|\to+\infty,\label{2-c}\\
\h\ff\to0,\qquad\mbox{as}\quad |y|\to+\infty.\label{3-c}
\end{gather}
\end{proposition}
\begin{proof}
It is straightforward to see that if $\ff$ is a nontrivial solution of \eqref{shrira-0} satisfying \eqref{1-c}-\eqref{3-c}, then
\begin{equation}\label{van}
\int_\rr\h\ff(x,y)\;\dd x=0.
\end{equation}
On the other hand, $\h\ff=\h k\ast f(\ff)$, where
\[
\widehat{\h k}(\xi,\eta)=-\ii\frac{\xi}{|\xi|+\xi^2+\eta^2}.
\]
By an argument similar to Lemma \ref{d-kernel}, there holds
\[
\h k(x,y)=\sqrt{\pi}\int_0^{+\infty}t^{5/2}\ee^{-t}
\left(t^2x^2+\left(t^2+y^2\right)^2\right)^{-\frac{3}{2}}
\sin\left(\frac{3}{2}\arctan\left(\frac{t|x|}{t^2+y^2}\right)\right)\;\dd t.
\]
The function $\h k$ does not change the sign, since
\[
\sin(\frac{3}{2}\arctan(x))=\frac{\sqrt{2}}{2}\frac{(1+(1+x^2)^{1/2})^{1/2}}{(1+x^2)^{3/4}}\left(2+(1+x^2)^{1/2}\right)>0.
\]
The proof then follows because if $\ff$ is positive, $\h\ff=\h k\ast f(\ff)$ has a definite sign, contradicting \eqref{van}.
\end{proof}

\section*{Acknowledgment}
The second author is partially supported by CNPq-Brazil and FAPESP-Brazil. The authors would like to thanks F.H. Soriano for the helpful discussion concerning the construction of the extension operator.

\end{document}